\documentclass[12pt]{amsart}

\usepackage{graphicx,color,latexsym}
\usepackage{verbatim}
\usepackage{enumerate}
\usepackage[all,cmtip]{xy}

\usepackage{subfigure}
\usepackage{mathrsfs}
\usepackage{amscd, amssymb, amsmath, amsthm, graphics}
\usepackage{amsmath,amsfonts,amsthm,amssymb}
\usepackage{latexsym,amsmath}
\usepackage{graphicx}

\usepackage{latexsym}
\usepackage{enumerate}
\usepackage{color}

\newtheorem{theorem}{Theorem}[section]
\newtheorem{prop}[theorem]{Proposition}
\newtheorem{lemma}[theorem]{Lemma}
\newtheorem{remark}[theorem]{Remark}

\newtheorem{definition}[theorem]{Definition}
\newtheorem{cor}[theorem]{Corollary}

\newtheorem{speculation}[theorem]{Speculation}
\newtheorem{example}[theorem]{Example}

\newtheorem{lem}[theorem]{Lemma}

\newtheorem{condition}[theorem]{Condition}

\theoremstyle{definition}

\newcommand{\wt}[1]{\widetilde{#1}}
\newcommand{\wh}[1]{\widehat{#1}}

\newcommand{\bthm}{\begin{theorem}}
\newcommand{\ethm}{\end{theorem}}

\newcommand{\blem}{\begin{lem}}
\newcommand{\elem}{\end{lem}}

\newcommand{\bcor}{\begin{cor}}
\newcommand{\ecor}{\end{cor}}

\newcommand{\bprop}{\begin{prop}}
\newcommand{\eprop}{\end{prop}}

\newcommand{\brmk}{\begin{remark}}
\newcommand{\ermk}{\end{remark}}

\newcommand{\bpf}{\begin{proof}}
\newcommand{\epf}{\end{proof}}

\newcommand{\beq}{\begin{equation}}
\newcommand{\eeq}{\end{equation}}

\numberwithin{equation}{section}

\def\C{\mathbb{C}}
\def\L{\mathbb{L}}

\def\Q{\mathbb{Q}}
\def\R{\mathbb{R}}
\def\Z{\mathbb{Z}}
\def\T{\mathbb{T}}
\def\CP{\mathbb{CP}}

\def\cC{\mathcal{C}}

\def\cJ{\mathcal{J}}

\def\cL{\mathcal{L}}

\def\cN{\mathcal{N}}

\def\w{\omega}

\def\xkm2{\overline{X}_{k-2}}

\begin{document}

\title[Stability and Existence of Surfaces]{Stability and Existence of Surfaces in   Symplectic 4-Manifolds with $b^+=1$}
\author{Josef G. Dorfmeister}
\thanks{JD was partially supported by the Simons Foundation $\#$246043.}
\address{Department of Mathematics\\ North Dakota State University\\ Fargo, ND 58102}
\email{josef.dorfmeister@ndsu.edu}
\author{Tian-Jun Li}
\thanks{TJL was partially supported by was partially supported by NSF Focused Research Grants DMS-0244663 and NSF grant DMS-1207037}
\address{School  of Mathematics\\  University of Minnesota\\ Minneapolis, MN 55455}
\email{tjli@math.umn.edu}
\author{Weiwei Wu}
\thanks{WW was partially supported by NSF Focused Research Grants DMS-0244663 and AMS-Simons travel funds.}
\address{Department of Mathematics\\  Michigan State University\\ East Lansing, MI 48910}
\email{wwwu@math.msu.edu}

\begin{abstract}
We establish various stability results for symplectic surfaces
in symplectic $4-$manifolds with $b^+=1$.     These results are then applied to prove the existence of representatives of Lagrangian ADE-configurations as well as to classify  negative symplectic spheres  in symplectic $4-$manifolds with $\kappa=-\infty$.  This involves the explicit construction of spheres in rational manifolds via a new construction technique called the tilted transport.
\end{abstract}
\maketitle

\tableofcontents

\section{Introduction}

Given a symplectic manifold $(M, \w)$, it is natural to
ask whether  a homology/cohomology class $A$ is represented by an embedded
symplectic (Lagrangian) submanifold.
Even with the various advanced techniques currently available and emerging nowadays, this remains a very difficult
question.  There are two particularly significant techniques in this direction:  for the classes $l[\w]$ when $[\w]$ has integral period and $l$ is a sufficiently high multiple, a general
existence was   obtained by  Donaldson's asymptotic holomorphic section
theory;
for a homology class $A\in H_2(M,\mathbb Z)$ which  is Gromov-Witten effective,  the  pseudo-holomorphic curve
machinery often produces embedded  symplectic representatives in this
class.   In dimension $4$,  Taubes' symplectic Seiberg-Witten theory  \cite{T,T1,T4} is especially powerful to establish the GW effectiveness.

In the current paper, we investigate cases  in dimension $4$
not covered by
the techniques mentioned above, e.g. we consider GW non-effective (or more precisely,
not necessarily GW effective) classes.
In fact, we approach this problem by answering a
natural extension that is closely related but rarely seen in the
 literature:

\begin{center}If there exists $V\subset M$ which is an $\w$-symplectic submanifold, can
$V$ be ``propagated" to other symplectic forms $\wt\w$?\end{center}

Such ``propagation" can be interpreted in different senses.  For
example, when $\wt\w$ is isotopic to $\w$, this problem has no new
content due to Moser's theorem. The main case we consider is when
$\wt\w$ is \textit{deformation equivalent} to $\w$, that is, when
there is a smooth family of symplectic forms $\{\w_t\}$ such that
$\w_0=\w$ and $\w_1=\wt\w$. When such propagation holds, we say
$(M,V,\w)$ possesses the \textit{stability property}.

In this paper we  establish  several stability results for connected
symplectic surfaces in  symplectic manifolds $(M, \omega)$ with
$b^+=1$, which allow us to address also the existence problem in
rather general settings. To explain further our results, we first
introduce some notions.

\begin{definition}\label{d:ConfR} Consider a graph $G$, where
each vertex $v_i$ is labelled by an element $A_i\in H_2(M,\Z)$, and
we denote $a_{ij}:=A_i\cdot A_j\ge0$.  Two vertices are  connected by edges
labeled by positive integers which sum up to $a_{ij}$ when
$a_{ij}\ne0$ . Assume that $|G|<\infty$.  We will refer to $G$ as a
\textbf{homological configuration}. Let $\omega$ be a symplectic
structure on $M$.
\begin{enumerate}
\item $G$ is called {\bf simple} if all labels on the edges are $1$.
\item $\w$ is called {\bf $G$-positive} if $\w(A_i)>0$ for all $i\le |G|$.

\item A curve configuration $V=\bigcup_{i=1}^{|G|}V_i$ is a
{\bf realization of the homological configuration $G$} with
respect to $\w$, if it consists of the following:

\begin{enumerate}
\item a one-one correspondence from the vertices $\{v_i\}$ of $G$ to an embedded $\w$-symplectic curve
$V_i\subset M$, $[V_i]=A_i$ for each $i\le |G|$ where $A_i$ is the
homology class labeled on $v_i$;
\item a one-one correspondence from $V_i\cap V_j$ to the edges connecting $v_i$ and $v_j$, and
the intersection multiplicity equals the marking on the
corresponding edges;
\item $V_i\cap V_j\cap V_k=\emptyset$ for all distinct $i,j,k$ and
\item there exists an almost complex structure $J$ compatible with $\omega$ making each $V_i$ $J$-holomorphic simultaneously.
\end{enumerate}

\end{enumerate}
\end{definition}

Notice that the curve configuration need not be connected.  Moreover, the last condition ensures that all intersections of components of $V$ are isolated and positive.

We consider stability for such configurations.
\begin{definition}
A curve configuration $V$ realizing $G$ with respect to $\omega$ is called {\bf $\omega$-stable} if for any
$G$-positive symplectic form $\tilde\w$ deformation equivalent to $\w$,
 there is a curve configuration $\tilde V$ realizing $G$ with respect to $\tilde\w$.
\end{definition}

   In some cases the relation
between $\tilde V$ and $V$ can be made more precise.  The following
is the main stability result.

\begin{theorem}\label{stability}Let $(M,\omega)$ be a symplectic manifold with $b^+=1$ and $G$ a
homological configuration represented by  a curve configuration $V$.
Then $V$ is $\omega$-stable. Moreover, $\tilde V$ can be chosen to be smoothly isotopic to $V$.
\end{theorem}

 At the core of Theorem \ref{stability}  are the existence  and abundance of  positive self-intersection
 symplectic surfaces along which inflation is carried out.
 The major source of such surfaces is Taubes' symplectic Seiberg-Witten theory.
 Moreover, the methods employed to prove Theorem \ref{stability} are rather robust and
 allow extensions in a number of directions.  We describe details in Sections \ref{s:tech} and \ref{nodal}.

In the rest of the paper we consider two applications of the
stability result Theorem \ref{stability}.  The first one is to show the
following:

\bcor\label{c:ADELag} In rational or ruled manifolds, any
homological Lagrangian ADE-configuration $\{l_i\}_{i=1}^n$ admits a
Lagrangian ADE-configuration representative. In the case of
$A_n$-configurations, one may require the configuration lie in
$M\backslash D$, where $D$ is a symplectic divisor disjoint from a
set of embedded symplectic representatives of the exceptional classes $\{E_i\}_{i=1}^{n+1}$\ecor

The definition of a \textit{homological Lagrangian configuration} is
given in Section \ref{s:ADE}.  For more general symplectic manifolds with $b^+=1$ we have:

\bcor\label{c:An-b1} Given a non-minimal symplectic $4$-manifold $(M,\w)$ with
$b^+=1$ and a set of exceptional classes $\{E_i\}_{i=1}^{n+1}$ where $\w(E_i)$ are
all equal.  Then there is a Lagrangian $A_n$-configuration of class
$\{E_i-E_{i+1}\}_{i=1}^n$.

\ecor

It is very tempting to assert the above corollary also holds for general symplectic
$4$-manifolds.  But there is a (possibly technical) catch: in general we do not know whether
two ball embeddings in a general symplectic $4$-manifold are connected.  That means
it is possible that two symplectic blow-up forms are not even symplectomorphic.  Therefore
one needs to be more precise when performing symplectic blow-ups on these manifolds.

Recall from Biran's stability of ball-packing in dimension $4$ (\cite{Bi99}) that  for any symplectic $4-$manifold $(M,\w)$,
there exists a number $N_0(M,\w)$
so that one may pack $N$ balls with volume less than
$vol_{\w}(M)/N_0$, as long as $[\w]\in H^2(M,\Q)$. The packing is
constructed away from a \textit{isotropic skeleton} defined in
\cite{Bi01}.  We show that $A_n$-type configurations still exist
 when these packed symplectic balls are blown-up.

\bcor\label{c:An} Given a symplectic $4$-manifold $(M,\w)$ with
$[\w]\in H^2(M,\Q)$ and a symplectic packing of $n+1\le N_0$
symplectic balls with equal volume $\le vol_\w(M)/N_0(M,\w)$.  Then there
is a Lagrangian $A_n$-configuration in
$(M\#(n+1)\overline{\CP^2},\w')$, where $\w'$ is obtained by blowing
up the embedded symplectic balls when the packing is supported
away from Biran's isotropic skeleton.

\ecor

\brmk

The more interesting part of this series of corollaries lies in the
case when the packing of $M$ is very close to a full packing. In
such scenarios, the geometry of the packing is usually difficult to
understand in an explicit way.  In particular, it would be very
difficult to construct these Lagrangian spheres by hand.  In
contrast, our theorem does not only guarantee the existence of the
Lagrangian spheres (which already appeared in \cite{LW}), we also
have control over their geometric intersection patterns, which is
usually difficult for Lagrangian or symplectic non-effective
objects.  From our proof, one may also conclude the existence of
symplectic ADE-plumbings when the involved classes have positive
symplectic areas (in fact, this is much easier because we do not need to
involve conifold transitions and can easily be extended to many
other types of plumbings). We leave the details for interested
readers.

\ermk


As another application, we consider the classification of negative self-intersection
spherical classes in symplectic rational or ruled surfaces.  This is
of interest for many different reasons: on the one hand, solely the
problem of existence of symplectic rational curves is already an
intriguing question when the corresponding class is not GT-basic,
which means its Gromov-Taubes dimension is less than zero.  However,
such curves are exactly the most interesting objects in many areas
of research.  For example, they span the Mori cone in birational
geometry, which has been extended into the symplectic category.

Moreover, such rational curves and the configurations they
form are crucial for various constructions in symplectic geometry (
for a very incomplete list, see  \cite{FS2},  \cite{S, S2},  \cite{SSW}, \cite{AP}, \cite{P2}).
As we will describe, we have found (-4)-symplectic
spheres in $\mathbb CP^2\#10\overline{\mathbb CP^2}$ along which a rational
blow-down incurs exotic examples of $E(1)_{2,k}$.  This
will be exploited further in our upcoming work [DLW].

In our approach, we also solved the problem of classifying homology
classes of \textit{smooth} embedded $(-4)$-spheres in rational
manifolds, which, to the best of authors' knowledge, is also new to
the literature.

To state our result,  denote the geometric automorphism group by
  \begin{equation}\label{e:D(M)}D(M)=\{\sigma\in Aut(H_2(M,\mathbb Z)): \sigma=f_*\text{ for some }f\in \textit{Diff}^+(M)\}.\end{equation}
Two classes $A,B\in H_2(M,\mathbb Z)$ are called {\it $D(M)$-equivalent} if there is a $\sigma\in D(M)$ such that $\sigma(A)=B$.

The following is our main result in Section \ref{s-4}:

\begin{theorem}[Classification of $-4$-spheres]\label{t:spheres}   Let
 $(M,\omega) $ be  a rational symplectic surface, i.e. $M=\mathbb CP^2\#k\overline{\mathbb CP^2}$ and $\{H,E_1,...,E_k\}$ the standard basis of $H_2(M,\mathbb Z)$.
 Consider any class $A\in H_2(M,\mathbb Z)$ with $A\cdot A=-4$.

\begin{itemize}

\item (Smooth case) $A$ is represented by a smooth sphere if and only if  $A$ is D(M)-equivalent to one class in the following list
   \begin{enumerate}

      \item $-H+2E_1-E_2$
      \item $H-E_1-..-E_5$
      \item $-a(-3H+\sum_{i=1}^9 E_i)-2E_{10}$ for some $a\in \mathbb N$ and $a\ge 2$
      \item $2E_1$
      \item $2(H-E_1-E_2)$

   \end{enumerate}

\item (Symplectic case) $A$ is represented by an
 $\omega-$symplectic sphere if and only if $A$ is represented by a
 smooth sphere, $[\omega]\cdot A>0$ and $K_\omega\cdot A=2$ ($K_\omega$ is the symplectic canonical class associated to  $\omega$).

 Moreover, up to $D(M)$-equivalence, the class $A$
 is one of the following:

 \begin{enumerate}[(i)]

   \item If $A$ is characteristic, then $k=5$ and $A$ is equivalent to type (2) above.
   \item If $A$ is not characteristic, then it is equivalent to either type (1) or type (3) above.
    \end{enumerate}
  \end{itemize}

\end{theorem}

Notice that for large enough $k$, some of the classes listed above are in fact pairwise $D(M$)-equivalent.
For completeness we will also give an overview of symplectic spheres
of square $-1,-2,-3$ in rational manifolds.  For those of squares
$-1,-2$ the results are essentially contained in the earlier works
\cite{TJLL} , \cite{L1} and \cite{LW}.  We also provide an explicit
algorithm in Remark \ref{r:algorithm} to implement our results.

Our result should be considered preliminary as it leads to more
interesting questions in two rather different directions. On the one
hand, recall the bounded negativity conjecture asserts that any
algebraic surface in characteristic zero has $C^2\ge n_X$ for any
prime divisor $C\subset X$ and some fixed $n_X\in\mathbb Z$. (for
accounts on this conjecture in complex geometry, see for example
\cite{H}, \cite{BH}, see also \cite{BB} for variations on this
conjecture).  Note that even for rational manifolds, only certain
small ones are known to satisfy this conjecture.

The conjecture makes perfect sense in the
symplectic category, that is, whether squares of symplectic curves in a given symplectic manifold are bounded from below.  For example, Lemma \ref{irr} partially reproduces the boundedness result of
Prop. 2.1 in \cite{BH}.  The computation relying on symplectic genus in \ref{s:-4red} shows some preliminary dichotomy patterns for a
negative curve: it is either not reduced (Section \ref{s-4}) and can be understood in small blow-ups, or its class is reduced
but only appears for a relatively large number of blow-ups.  For example, our result says $(-4)$-spheres can either be equivalent
to a curve in two blow-ups, or its class may only appear when the blow-up number hits $10$.
It seems reasonable to speculate that this continues to hold at least for
$(-n)$-spheres-there might be more complicated classes that are not blow-ups of classes in our list, but
they only show up when the blow-up numbers are large enough, hence for a fixed symplectic rational manifold, there
are only finitely many such hierarchies.  If this could be verified, one could possibly approach the case of higher genus
with similar methods.  This will be investigated in future work.


On the other hand, the existence part of Theorem \ref{t:spheres}
requires a new technique, which we call the \textit{tilted
transport}. This is very similar to the usual parallel transport
construction of Lagrangian submanifolds, but due to the ``softer"
nature of symplectic objects, this construction is also much more
flexible and even can be formulated quite combinatorially.  Moreover,
this simple technique could lead to the construction of a wealth of
symplectic submanifolds that are not GW-effective out of a Lefschetz
fibration, thus should be of independent interest. We describe this
construction in Section \ref{s:Con-4} and apply it to construct
$(-4)$-spheres in our classification.


Similar to the case of rational manifolds, we obtain for irrational
ruled manifolds the following more complete classification.  A
corresponding characterization in the smooth category appears as Lemma
\ref{irr}.

\begin{theorem}\label{irratspec}
Suppose $(M, \omega)$ is an irrational ruled $4-$manifold. Let $A\in
H_2(M,\mathbb Z)$ with $\omega(A)>0$.
  Then $A$ is represented by a connected $\omega-$symplectic sphere if and only if
$A$ is represented by a connected smooth sphere and
$g_{\omega}(A)=0$.

Moreover, suppose  $A$ is represented by a connected $\omega$-symplectic sphere. Then
\begin{enumerate}
\item $A\cdot A \geq 1 -  b^- (M)$.

\item   $A$ is characteristic only if $A \cdot A = 1- b^-(M)$, and $A$ is $D(M)-$equivalent to $E_1-E_2-\cdots -E_{1-b^-(M)}$.

\item   If $A$ is not characteristic, then $A$ is $D(M)-$equivalent to $F-E_1-\cdots- E_{l}$ for $l=-A\cdot A$.

\end{enumerate}
%
%
Moreover, when $A$ is not characteristic, then it is the blow-up of
an exceptional sphere.
\end{theorem}

 As stated above,
one motivation for investigating stability is to find a general
existence criterion for connected embedded symplectic surfaces in a
given homology class (as discussed in the survey \cite{L2}, see also \cite{LZ}) although
we have generalized the context to symplectic configurations. The results above lead us to offer the following speculation.

\begin{speculation}\label{spec:isphere}Let $(M,\omega) $ be a symplectic $4$-manifold.
  Let $A\in H_2(M,\mathbb Z)$ be a homology class.
  Then $A$ is represented by a connected $\omega$-symplectic surface if and only if \begin{enumerate}
\item $[\omega]\cdot A>0$,
\item $g_\omega(A)\geq 0$ and
\item $A$ is represented by a smooth connected surface of genus $g_{\omega}(A)$.
\end{enumerate}

\end{speculation}

Theorem \ref{t:spheres} thus verifies this speculation for spheres
in rational manifolds  with $A^2\ge-4$ and Theorem \ref{irratspec}
for all spheres in irrational ruled surfaces. This is in a sense by
``brutal force": we give a complete classifications of the smoothly
representable and symplectically representable classes and compare
them.  It would be interesting to have a construction independent of
these classification results.\\


\noindent\textit{Outline of the paper:} In Section \ref{s:tech} and
\ref{nodal} we establish the stability result \ref{stability}.
Section \ref{s:tech} contains some technical tools useful for
finding symplectic submanifolds and inflations adapted from
\cite{DL} to the configuration case.  In Section \ref{s:ADE} we
consider the stability and existence for Lagrangian configurations.
Section \ref{s-4} classifies $(-4)$-spheres in both smooth and
symplectic categories, with a subsection specifically devoted to
tilted transports.  Section \ref{class} completes the discussion for
symplectic manifolds with
$\kappa(M)=-\infty$ by considering irrational ruled manifolds.\\

\noindent\textit{Notation:} Let $M=\mathbb CP^2\#k\overline{\mathbb CP^2}$.  We use the  standard basis for
 $H_2(M,\mathbb Z)$ given by $\{H,E_1,\dots,E_k\}$.
 Denote by $K_{st}=-3H+\sum_{i=1}^k E_i$ the standard canonical
 class.
 Similarly we use the standard basis $\{A,B\}$ for $H_2(S^2\times S^2,\mathbb Z)$.

 For non-minimal irrational ruled surfaces $M$, we use $\{S,F,E_1,\dots, E_n\}$ as the basis of $H_2(M,\Z)$,
 where $S$ denotes the class of the base and $F$ the class of the
 fiber.\\

\noindent\textit{Acknowledgements:}  The third author is grateful to Ronald Fintushel for introducing him to the problem of bounded negativity, and Kaoru Ono for explaining patiently many details regarding
 conifold transitions.  He would also like to  warmly thank Selman Akbulut for
 his interest in this work and offering an opportunity to present it in the Topology seminar at MSU.

\section{A Technical Existence Result}\label{s:tech}

In this section we wish to extend Theorem 2.13, \cite{DL}, to the
more complicated curve configurations of Def \ref{d:ConfR}.  The key
to the proof of Theorem 2.13 is Lemma 2.14 therein, which provides
for the existence of a curve in a given class $A\in H_2(M;\mathbb
Z)$ under certain restrictions on $A$.  At the core of the proof of
this lemma are results on the existence of a suitably generic almost
complex structure among those making a fixed submanifold $V$
pseudoholomorphic such that classes $A$ with negative Gromov
dimension are not represented by pseudoholomorphic curves.


Consider a curve configuration $V=\cup_{i=1}^k V_i$. Let $\mathcal
J_{V_i}$ denote the set of almost complex structures compatible with
$\omega$ and making $V_i$ pseudoholomorphic and let $\mathcal
J_V=\cap_i\mathcal J_{V_i}$. Notice that $\mathcal J_V\ne\emptyset$
by Def \ref{d:ConfR}.

\subsection{Generic Almost Complex Structures}

We begin by defining a universal space which we shall use throughout
this section: Fix a closed compact Riemann surface $\Sigma$. The
universal model $\mathcal U$ is defined as follows: this space will
consist of Diff$(\Sigma)$ orbits of a 4-tuple $(i,u,J,\Omega)$ with
\begin{enumerate}
\item $u:\Sigma\rightarrow M$ an embedding off a finite set
of points from a Riemann surface $\Sigma$ such that
$u_*[\Sigma]=A\in  H_2(M,\mathbb Z)$ and $u\in W^{k,p}(\Sigma,M)$ with $kp>2$,

\item $\Omega\subset M$ a set of $k(A)=\frac{1}{2}(A\cdot A-K_\omega\cdot A)$ distinct points (with
$\Omega=\emptyset$ if $k(A)\le 0$) such that $\Omega\subset
u(\Sigma)$,

\item $i$ a complex structure on $\Sigma$ and $J\in \mathcal J_V$.
\end{enumerate}

Note that every map $u$ is locally injective, and one has a
fibration $\pi: \mathcal U\to \cJ_V$. Moreover, in order for the set
$\mathcal U$ to be of any interest, it is natural to implicitly
assume that $A\cdot [V_i]\ge 0$ for all $i\in I$ unless $A=[V_i]$
and $[V_i]\cdot [V_i]<0$.

The goal of this section is to show that for a sufficiently generic
choice of almost complex structure in $\cJ_V$ the fiber in $\mathcal
U$ either has the expected dimension or dim $\ker(\pi)=0$.  We must
distinguish two cases: If $A\ne[V_i]$, then for any point in
$\mathcal U$, $u(\Sigma)$ will contain a point not in $V$.  If
$A=[V_i]$ for some $i$, then in $\mathcal U$ we will distinguish the
embedding of $V_i$ (and possibly $V_j$ if $[V_i]=[V_j]$) from the
other points in $\mathcal U$.

It should be noted that we prescribe $V$, hence the manifolds comprising $V$ may be very poorly behaved with respect to Gromov-Witten moduli.  In particular, $k([V_i])<0$ is possible.

Consider first a class $A\ne[V_i]$.  Note that this includes, for example, the class $A=[V_i]+[V_j]$.
For such a class, any element $u\in \mathcal U$ will have a point $x_0\in\Sigma$ such that $u(x_0)\not\in V_i$.
Therefore, the proof of Lemma A.1, \cite{DL}, applies as written there, albeit with a different set of underlying almost complex structures.

\begin{lem}\label{genericA}
Let $A\in H_2(M,\mathbb Z)$, $A\ne [V_i]$ for any $i$ and
$k(A)\ge 0$. Let $\Omega$
denote a set of $k(A)$ distinct points in $M$. Denote the
set of pairs $(J,\Omega)\in \mathcal J_V\times M^{k(A)}$ by
$\mathcal I$. Let $\mathcal J_V^A$ be the subset of pairs
$(J,\Omega)$ which are nondegenerate for the class $A$ in
the sense of Taubes \cite[Def. 2.1]{T4}. Then $\mathcal
J_V^A$ is a set of second category in $\mathcal I$.
\end{lem}

We remind the readers that the nondegeneracy involved in the above lemma pertains only to embedded curves.
Consider now the case $A=[V_i]$ for some $i$ (for simplicity assume $i=1$).    In this case, as noted above, $\mathcal U$ consists of two types of points:  Those which are embeddings of components of $V$ and those which contain a point not in $V$.  We concentrate first on the points corresponding to embeddings of components of $V$.

Let $j_i$ be an almost complex structure on $V_i$ and denote ${ j}=(j_1,..,j_k)$. Define
\[
\mathcal J_V^{ j}\,{=}\,\{J\in\mathcal J_V\vert J\vert_{V_i}=j_i\}\!
\]
and call any $J$-holomorphic embedding for $J\in
\mathcal J_V^{j}$ a $j$-holomorphic embedding.  Notice that
\[
\mathcal J_V=\cup_{j}\mathcal J_V^{j}
\]
and by assumption $\mathcal J_V\ne \emptyset$.  Hence for some $j$ , $\mathcal J_V^{j}\ne\emptyset$.

Consider the behavior of the linearization of $\overline\partial_{I,J}$ at a point $u\in\mathcal U$ such that $u:(\Sigma,I)\rightarrow
(M,J)$ is a $j$-holomorphic embedding of $V_1$.

\begin{lem}\label{genericV}
Let $A=[V_1]$ and fix $j$.  Fix a $j$-holomorphic embedding
$u:(\Sigma,I)\rightarrow (M,J)$ of $V_1$ (or of $V_j$ if
$[V_1]=[V_j]$).
\begin{enumerate}
\item If $k(A)\ge 0$, then there exists a set $\mathcal G_V^{j}$ of
second category in $\mathcal J_V^{j}$ such that for any
$J\in \mathcal G_V^{j}$ the linearization of $\overline
\partial_{i,J}$ at the embedding $u$ is surjective.
\item If $k(A)<0$, then there exists a set $\mathcal
G_V^{j}$ of second category in $\mathcal J_V^{j}$ such
that for any
$J\in \mathcal G_V^{j}$ the linearization of $\overline
\partial_{i,J}$ at the embedding $u$ is injective.
\end{enumerate}
\end{lem}

The proof of this lemma follows exactly as the proof of Lemma A.2, \cite{DL}, as the necessary perturbations occur in a neighborhood of a point on $V_1$ which is not contained in any other $V_j$.  Our conditions ensure that such a point exists.

The key point of Lemma 2.2 is that, in spite of the non-genericity of almost complex structure in $\cJ_V$, we may at least require that
non-generic curves do not have nontrivial deformations.  This is recapped in the following:

\begin{lem}\label{jV}Assume $A=[V_i]$ for some $i$.  Let $\Omega$ denote a set of $k(A)$ distinct
points in $M$.
\begin{enumerate}
\item $k(A)\ge 0$: Denote the set of pairs
$(J,\Omega)\in \mathcal J_V\times M^{k(A)}$
by $\mathcal I$. Let $\mathcal J_{[ V]}$
be the subset of pairs $(J,\Omega)$ which are
nondegenerate for the class $A$ in the sense of Taubes
\cite{T4}. Then $\mathcal J_{[ V]}$ is dense in
$\mathcal I$.

\item $k(A)<0$: There exists a set of second category
$\mathcal J_{[ V]}\subset \mathcal J_V$ such that there exist no pseudoholomorphic
deformations of $V$ and there are no other
pseudoholomorphic maps in class $A$ except possibly components of $V$.
\end{enumerate}
\end{lem}

The general tactic for a proof of this lemma is as follows:  Recall that
\[
\mathcal J_V\times M^{k(A)}=\sqcup_{j}\left(\mathcal J_V^{j}\times M^{k(A)}\right).
\]
Now find a dense subset of $\mathcal J_V^{j}\times M^{k(A)}$ (when non-empty).  Lemma \ref{genericV} provides for a suitable subset $\mathcal G_V^{j}\subset \mathcal J_V^{j}$ of second category for each embedding of a component of $V$ ensuring that the differential operators at such an embedding have the appropriate behavior.  Taking the intersection of all such sets produces a set $\mathcal G_V$ which is still of second category in $\mathcal J_V^{j}$.  Now consider only almost complex structures $J\in \mathcal G_V$ to understand the behavior of $\mathcal U$ at points which have a point in the image off of $V$.  The methods of the proof of Lemma A.1, \cite{DL} apply in this setting.

%


\subsection{Existence of Symplectic Submanifold}

In this section we state and justify a result analogous to Lemma
2.14, \cite{DL}.   Let $V$ be a realization of some homological
configuration $G$. Note that Lemma 2.14, \cite{DL}, can be used to
provide a $\omega$-symplectic submanifold intersecting some $V_i$ as
needed, however it is not immediatley clear why this curve must
intersect the other $V_i$ also locally positively and transversally.
In particular, the restriction of almost complex structures from
$\mathcal J_{V_i}$ to $\mathcal J_V$ must be justified.  This has
been prepared in the previous section and at all points in the proof
of Lemma 2.14, \cite{DL}, these results should be inserted.

We begin with the following observation.

\begin{lemma}\label{posA}
Let $(M,\omega)$ be a symplectic manifold with $b^+(M)=1$, $W$  a connected embedded symplectic submanifold and $A\in H_2(M,\mathbb Z)$.  Assume that $(A-K_\omega)\cdot [W]>0$, $A\cdot A\ge 0$ and $A\cdot [\omega]\ge 0$.  Then
\begin{enumerate}
\item $A\cdot [W]\ge 0$ and
\item if $A\cdot [W]=0$, then either $[W]\cdot[W]=0$ and $A=\lambda[W]$ up to torsion or $W$ is an exceptional sphere.
\end{enumerate}
\end{lemma}

\begin{proof}
If $[W]\cdot [W]\ge 0$, then $A\cdot [W]\ge 0$ by the light cone lemma (Lemma 3.1, \cite{TJLL}).

Now consider $[W]\cdot [W]<0$.  Let $(A-K_\omega)\cdot [W]>0$ and assume further that $A\cdot [W]<0$.  Then
\[
K_\omega\cdot [W] <A\cdot [W]<0.
\]
As $W$ is a connected embedded symplectic submanifold, it satisfies the adjunction equality which implies
\[
[W]\cdot [W]+2-2g=-K_\omega\cdot [W]>0
\]
and thus $[W]\cdot [W]>2g-2$.  Therefore $[W]\cdot [W]\ge 0$ unless $g=0$ and $[W]\cdot [W]=-1$.
%
Thus $[W]$ is  an exceptional sphere.  Then $-1=K_\omega\cdot [W] <A\cdot [W]<0$, contradicting $A\cdot [W]<0$.

This proves the non-negativity statement of the lemma.

Assume that $A\cdot [W]=0$.  Then by the light cone lemma $[W]\cdot [W]\le 0$.   If $[W]\cdot[W]=0$, then, again by the light cone lemma, $A=\lambda[W]$ up to torsion.  Otherwise $W$ is an exceptional sphere.

\end{proof}

The following is a version of this statement for exceptional spheres.

\begin{lemma}\label{posexc}
Let $(M,\omega)$ be a symplectic manifold with $b^+(M)=1$, $W$  an embedded symplectic submanifold and $A\in \mathcal E_\omega$ an exceptional sphere.  Let $(A-K_\omega)\cdot [W]>0$.  Then
\begin{enumerate}
\item $A\cdot [W]\ge 0$ unless $A=[W]$ and
\item if $A\cdot [W]=0$, then there exists an exceptional sphere in the class of $A$ which is disjoint from $W$.
\end{enumerate}
\end{lemma}

\begin{proof}
Notice that $A$ is GT-basic as it is an exceptional sphere.

Assume that $[W]\cdot[W]\ge 0$.  Then for any almost complex
structure making $W$ pseudoholomorphic, we can find a connected
pseudoholomorphic representative for $A$.  This curve may have many
components connected by nodes, some multiply covered, but each image
curve must intersect $W$ locally positively.  In particular, this
representative of $A$ can have components covering $W$, however
these also contribute only positively to $A\cdot [W]$.  Therefore
$A\cdot [W]\ge 0$.

Let $[W]\cdot[W]<0$.  As in Lemma \ref{posA}  , the assumption
$(A-K_\omega)\cdot [W]>0$ and $A\cdot [W]<0$ implies that $W$ is an
exceptional sphere.   Lemma 3.5, \cite{TJLL}, ensures that $A\cdot
[W]\ge 0$ unless $A=[W]$.

Let $A\cdot [W]=0$.  If $[W]\cdot[W]\ge 0$, then by the above
argument a connected $J$-holomorphic representative of $A$ can be
found such that $A=\sum A_i+ m[W]$.

Since each component $A_i$ intersects $W$ non-negatively, by pairing
with $[W]$, one sees that $A_i$ are indeed disjoint from $W$.  The
connectedness assumption thus implies $m=0$. A standard genericity
argument shows by perturbing $J$ away from $W$ we may assume there
is only one component among $A_i$ which is non-empty, giving the
desired exceptional sphere. If $W$ is an exceptional sphere, then
Lemma 3.5, \cite{TJLL}, provides for the existence of a
representative of $A$ disjoint from $W$.

\end{proof}

The converse of these results, i.e. that $A\cdot [W]\ge 0$ implies
$(A-K_\omega)\cdot [W]>0$, need not be true.  Let $M=S^2\times
\Sigma_3$, $\Sigma_3$ a  genus 3 surface.  Consider the standard
basis of $H_2(M,\mathbb Z)$ and any symplectic form with
$K_\omega=4F-2S$.  Let $[W]=S-F$ and $A=S+F$.  Then $A\cdot [W]=0$
but $(A-K_\omega)\cdot [W]=-6$.  Notice that $[W]$ is representable
by a symplectic submanifold of genus 3 for some symplectic form with
this canonical class.

The following result is an extension of Lemma 2.14, \cite{DL} from
a submanifold to a curve configuration.  The proof is largely
identical hence we only give an outline with appropriate details
relevant to multiple components.

\begin{lem} \label{rel submanifold} Fix a symplectic form $\omega$
on $M$ with $b^+(M)=1$ such that  $V$ is a curve configuration.  For
any $A\in H_2(M;\mathbb Z)$ with
\begin{gather*}
A\cdot E >0\mbox{ for all } E\in \mathcal E_\omega,\\
 A\cdot A>0,\hspace{3mm} A\cdot [\omega]>0,\\
(A-K_{\omega})\cdot [\omega]>0,\hspace{3mm}
(A-K_{\omega})\cdot (A-K_{\omega})>0,\\ (A-K_{\omega})\cdot
[V_i]>0\mbox{ for all } i\in I,
\end{gather*}
there exists a connected embedded $\omega$-symplectic submanifold $C$ in the
class $A,$ intersecting $V$ $\omega$-orthogonally and positively.
\end{lem}

\bpf The assumptions
\begin{gather*}
A\cdot E >0\mbox{ for all } E\in \mathcal E_\omega,\\
 A\cdot A>0,\hspace{3mm} A\cdot [\omega]>0,\\
(A-K_{\omega})\cdot [\omega]>0,\hspace{3mm}
(A-K_{\omega})\cdot (A-K_{\omega})>0,
\end{gather*}
together with $b^+=1$ ensure that for generic almost complex structures $A$ admits a connected embedded pseudoholomorphic representative (see \cite{TJLL}).

By Lemma \ref{posA}, the assumption $(A-K_{\omega})\cdot[V_i]>0$ together with $A\cdot A>0$ and $A\cdot[\omega]>0$ ensures that $A\cdot[V_i]>0$ unless possibly if $V_i$ is an exceptional sphere.  In the latter case, $A\cdot [V_i]>0$ by assumption. Therefore the results of the previous section on the genericity of almost complex structures can be applied to $A$ and its components.

Standard arguments (see for example the proof of Lemma 2.14, \cite{DL}) lead to  the following decomposition for
the class $A$:
\[
A=\sum_{k([V_i])\ge 0}m_i[V_i]+\sum_{k([V_i])< 0}m_i[V_i]+\sum_ib_iB_i+\sum_ia_iA_i
\]
where
\begin{enumerate}
\item $B_i\in \mathcal E_\omega$,
\item  all intersections of distinct classes are non-negative,
\item $A_i\cdot A_i\ge 0$ and $k(A_i)\ge 0$ and
\item all sums are finite.
\end{enumerate}

Assume that the second summand (over $k([V_i])< 0$) is empty.  Then consider any $V_i$ with $A\cdot [V_i]>0$ and fix this as $V_i=V_0$ ( in particular, if $A=m[V_i]$, choose this particular $i$).  Move all of the other $[V_i]$-terms into either the $B_i$ or the $A_i$ summand as appropriate (this can be done as $k([V_i])\ge 0$).  Now repeat the argument in Case 1 of the proof of Lemma 2.14, \cite{DL}.  As in that proof, if $A\ne m[V_0]$ or $A=m[V_0]$ and $k([V_0])>0$, then the proof is complete.  The remaining case is $A=m[V_0]$ and $k([V_0])=0$, which implies that either $m=1$ or $[V_0]^2=0$.  Notice that if $[V_0]^2=0$, then $A\cdot [V_0]=0$, contradicting our assumption that $A\cdot A>0$.    If $m=1$, then $A=[V_0]$ and thus $(A-K_\omega)\cdot [V_0]=2k([V_0])=0$, contradicting our assumptions.

This provides for an embedded $J$-holomorphic curve $\tilde C$ representing $A$ with a single non-multiply covered component intersecting $V_0$ positively where $J$ is now chosen appropriately from $\mathcal J_V$.  This however also implies that $\tilde C$ intersects all $V_i$ locally positively.  When $A\ne m[V_0]$, then also, by our choice of $V_0$, $A\ne n[V_i]$ for all $i$, and hence we can ensure that $\tilde C$ is distinct from any $V_i$.  When $k([V_0])>0$, then by an appropriate choice of points $\Omega_{k(A)}$ we can again ensure that $\tilde C$ is distinct from any $V_i$.

Each of the members of $V$ is distinct, hence all intersection points are isolated.  Apply Lemma 3.2 and Prop. 3.3, \cite{LU}, to perturb only $\tilde C$ to a pseudoholomorphic curve $C'$, while leaving each $V_i$ unchanged,  such that $C'$ intersects the pseudoholomorphic curve family locally positively and transversally.  Now perturb $C'$ further to a $J$-holomorphic curve $C$ which is $\omega$-orthogonal to $V$.  This involves a local perturbation around the intersection points and can be done in such a way as to ensure that distinct intersections stay distinct, see \cite{G}.

Assume now that the second summand is not empty.  We rewrite the class $A$ as follows.  First, move all of the terms in $\sum_{k([V_i])\ge 0}m_i[V_i]$ into either the $B_i$ or the $A_i$ summand, as done in the first case.  Secondly, distinguish in $\sum_{k([V_i])< 0}m_i[V_i]$ those classes corresponding to components of the curve in a class $m_i[V_i]$ which are not multiple covers of $V_i$ and those which are multiple covers.  As noted in \cite{DL}, the former all correspond to curves which underlie the genericity results discussed previously.  A generic choice of almost complex structure in $\mathcal J_V$ thus either removes such curves (if $k(m_i[V_i])<0$) or they can be included in the $A_i$ or $B_i$  sum (if $k(m_i[V_i])\ge 0)$.  Denote the remaining terms
\[
V_{mult}=\sum_{\substack{k([V_i])< 0,\\ V_i \;\; mult.\;\;cover}}m_i[V_i].
\]

Consider now $\tilde A=A-V_{mult}=\sum a_iA_i+\sum b_iB_i$.  The arguments in \cite{Bi99} or \cite{DL} continue to hold albeit with $mZ$ resp. $mV$ replaced by $V_{mult}$.  As in \cite{Bi99}, we obtain the estimate  $\sum k(A_i)\le k'(\tilde A)$.  Moreover,
\[
2k(A)-2k'(\tilde A)=\underbrace{(A-K_\omega)\cdot V_{mult}}_{>0\mbox{ by assumption}}+ \mbox{ non-negative terms}
\]

Hence $k(A)>\sum k(A_i)$ and thus either $A=V_{mult}$ or $V_{mult}=0$.  In the latter case the result follows from arguments as above.  If $A=V_{mult}$, then note that
\[
2k(A)=A\cdot A-K\cdot A=\sum_{\substack{k([V_i])< 0,\\ V_i \; mult.\;\;cover}} m_i(A-K)\cdot[V_i]>0.
\]
Thus choose a point not on $V$.  Then by Lemma \ref{genericA} and
\ref{jV} together with the assumptions on $A$ there  exist $(J,\Omega)$ such that $A$ is represented by
an embedded curve meeting $V$ locally positively, and, as before,
this curve can be made $\omega$-orthogonal to $V$ by the results in
\cite{LU} and \cite{G}.

\epf

\bcor\label{c:inflation} For any $\w$, $A$ and $V$ as in Lemma
\ref{rel submanifold}, there is a family of symplectic forms
$\{\w_t\}_{0\le t\le1}$ such that $\w_0=\w$,
$[\w_t]=[\w]+tPD([A])$ and $V$ is a curve configuration with respect to $\omega_t$.

\ecor

Notice that the conditions
\begin{equation}\label{e:conditions}\begin{gathered}
A\cdot E >0\mbox{ for all } E\in \mathcal E_\omega,\\
 A\cdot A>0,\hspace{3mm} A\cdot [\omega]>0\\
(A-K_{\omega})\cdot [\omega]>0,\hspace{3mm}
(A-K_{\omega})\cdot (A-K_{\omega})>0
\end{gathered}
\end{equation}

from Lemma \ref{rel submanifold} on $A$
ensure that $A$ has a $J$-holomorphic representative for any $J$. In
fact, they ensure that $A$ is a GT-basic class, see \cite{TJLL}.  Moreover, when $b^+=1$, this class is representable by a connected curve.

In particular, the proof above makes explicit use of only
$A\cdot [V_i]\ge 0$ and $(A-K_\omega)\cdot [V_i]>0$; any GT-basic classes satisfying these conditions will verify the lemma.

Specifically, if $A\cdot A=-1$, then it is necessary to assume that no component of $V$ has $[V_i]=A$ or a conclusion as in Lemma \ref{rel submanifold} must be false.  However, this is implied by the assumptions $(A-K_{\omega})\cdot
[V_i]>0$ and $A\ne[V_i]$  for all $ i\in I$, see Lemma \ref{posexc}.  Therefore the proof of Lemma \ref{rel submanifold} is immediately applicable, albeit with the slight change in Case 1 that $A=B_i$ is now allowed.  We state this as a lemma:


\begin{lem} \label{rel exceptional} Fix a symplectic form $\omega$ on $M$ such that $V$ is a curve configuration.  Assume that  $A\in\mathcal E_\omega$ and $A\ne [V_i]$ for any $i\in I$.  Furthermore, assume that
\begin{gather*}
 (A-K_{\omega})\cdot
[V_i]>0\mbox{ for all } i\in I.
\end{gather*}
Then there exists a connected embedded $\omega$-symplectic submanifold $C$ in the
class $A$ intersecting $V$ $\omega$-orthogonally and positively.
\end{lem}

\subsection{The Relative Symplectic Cone}  Let $\Omega(M)$ denote the space of orientation-compatible symplectic forms on $M$.
The symplectic cone $\mathcal C_M$ is the image of the cohomology class map
 \[
\begin{array}{ccc}
\Omega(M)&\rightarrow& H^2(M,\mathbb R)\\\omega&\mapsto &[\omega].
\end{array}
\]
For a smooth connected surface $V$, the relative symplectic cone $\mathcal C^V_M\subset \mathcal C_M$
is the set of classes of symplectic forms making $V$ a symplectic submanifold.
Since $V$ is $\omega-$symplectic,  by Theorem 2.13 in \cite{DL}, $\mathcal C^V_M
$ contains the cone
\[
\mathcal C_{M,
K_{\omega}}^{A}=\{\alpha=[\omega'] |\;   \omega' \mbox{ symplectic with  $K_{\omega'}=K_{\omega}$, } \alpha\cdot A>0 \},
\]
where $A=[V]$.

Now let $V=\cup_{i\in I} V_i$ be a collection of connected embedded
curves.  One may similarly define $\cC_M^V$. Let $K$ be a symplectic
canonical class for $M$ and define
\[
\mathcal D_K^V=\left\{[\omega]\in\mathcal C_M\;|\;\;[\omega]\cdot[V_i]>0\mbox{  for all }i\in I,\;K_\omega=K\right\}.
\]
This is the set of classes in the $K_\omega$-symplectic cone which
pair positively with each component in the curve configuration. By
definition,
\[
\mathcal D_K^V=\cap_{i\in I}\mathcal C^{[V_i]}_{M,K}.
\]

Note that this does not imply the existence of a symplectic form
$\omega$ making $V$ a curve configuration.

\begin{theorem}\label{relcone}
Let $M$ be a $4$-manifold with $b^+(M)=1$ and $\{V_i\}_{i\in I}$ a family of submanifolds of $M$ such that there exists a symplectic form $\omega$ on $M$ making $V=\cup V_i$ into a curve configuration.  Then
\[
\mathcal D_{K_\omega}^V\subset \mathcal C^V_M.
\]
In particular, for every $\alpha\in \mathcal D_{K_\omega}^V$ there exists a symplectic form $\tau$ in the class $\alpha$ making $V$ into a curve configuration.  Moreover, $\tau$ is deformation equivalent to $\omega$ through forms making $V$ a curve configuration.
\end{theorem}

\begin{proof}

Fix a symplectic form $\omega$ making $V$ into a curve
configuration.   We may assume that $[\omega]\in H^2(M, \mathbb Z)$: since making $V_i$ into a symplectic curve is an open condition and we have only finitely many components of $V$, we consider the intersection of these open sets.  In this intersection there must be a symplectic form $\beta$ with $[\beta]\in H^2(M,\mathbb Q)$ making $V$ into a curve configuration.  Now rescale to get $\omega$.

Let $e\in \mathcal D_{K_\omega}^V\cap H_2(M,\mathbb Z)$.  Then the
class $A=le-PD[\omega]$
satisfies the assumptions of Lemma \ref{rel submanifold} for
sufficiently large $l$.  Thus there exists an $\omega$-symplectic
submanifold $C$ intersecting $V$ locally positively and
$\omega$-orthogonally, with class $[C]=A$.

Let $N$ be an $S^2$-bundle over a surface of genus
$g_\omega(A)=\frac{1}{2}(A\cdot A+K_\omega \cdot A)+1$ and $S$ be a
section with $S\cdot S=-A\cdot A$.  Now apply the pair-wise sum of
Thm. 1.4, \cite{G}, to $(M,C)$ and $(N,S)$.  The thus generated
manifold $X$ is diffeomorphic to $M$.  This symplectic sum produces
a family of deformation equivalent symplectic forms $\omega_t$ on
$M$ in the class $[\omega]+tA$ for $t\ge 0$ such that $V$ is a
curve configuration with respect to $\omega_t$. Thus $[\omega_1]=le$
and the class $e$ is representable by a symplectic form making $V$
into a curve configuration.

Now repeat the argument as in the proof of Theorem 2.13, \cite{DL} for general $e\in \mathcal D_{K_\omega}^V$.

\end{proof}

\begin{remark}  The results of this section are similar to Theorem 1.2.7 (the second part),
1.2.12 and Corollary 1.2.13, \cite{M3} and it seems instructive to
compare the two situations for the reader's convenience.  The key
difference is in the conditions we imposed in Lemma \ref{rel
submanifold}, compared to the second part of Theorem 1.2.7 in
\cite{M3} (we do not have a clear idea about the relations between
the conditions therein and ours).

Our set of assumptions, directly adapted from \cite{DL}, offer
several aspects of convenience.  On the one hand, they replace
rational/ruled assumptions in Proposition 3.2.3 and 5.1.6 in
\cite{M3} so that we have results for manifolds with $b^+=1$
with our assumptions. On the other hand, since this set of
conditions is automatically satisfied when $A$ is a positive class
pairing with all components in the configurations positively and
raised to a high multiple, they are particularly suitable for
performing inflation and allows for slightly more flexibility.  Hence
our singular set places no restrictions on the intersections such as being
transverse between components.

Another difference between the two results stem from the almost
complex structures to be considered.  In this paper we assume that
there exists an almost complex structure making each $V_i$
pseudoholomorphic at the same time. This statement concerns only the
submanifolds themselves and the almost complex structure outside can
be generically chosen. In \cite{M3} the authors consider adapted
almost complex structures which places conditions on a fibered
neighborhood of the configuration $\mathcal S$.  This is crucial for
the geometric constructions therein, which was used to simplify
 the ``B" part of the curve in the decomposition (3.1.2) in
 \cite{M3}.

What we did not deal with in the current paper is the family
inflation (Thm 1.2.12, \cite{M3}), which probably requires similar techniques as in
\cite{M3}.

Furthermore,  the statement of Thm \ref{relcone} is the same as Prop. 1.2.15(i), \cite{M3}.  However, due to the differences in the sets under consideration, as described above, Thm \ref{relcone} is in a more general setting, allowing more general manifolds and configurations.  The proofs also differ slightly:  \cite{M3} use inflation along nodal curves,
we only need to consider embedded curves.

Both results are an extension of Thm. 2.13, \cite{DL}, which is for a single symplectic surface.

\end{remark}


\section{Stability of symplectic curve configurations}\label{nodal}

Suppose that $V$ is a symplectic surface in a symplectic 4-manifold
$(M,\omega)$.   Then we can consider the stability of this surface
under  (not necessarily continuous) variations of the symplectic
structure.  Our main result concerns deformations of the symplectic
structure.
This is the setting of Theorem \ref{stability} and is discussed in
\ref{21}.   Note that $V$ and $\tilde V$ are not just homologous,
but isotopic.

In nice cases, we also address the stability for arbitrary
symplectic structure $\tilde \omega$. The issue which arises in this
context is that it is in general not understood how to go from one
deformation class of symplectic structures to another.  Two cases in which this is explicitly understood are the following.

\begin{enumerate}

\item  The special family of manifolds with $\kappa=-\infty$: For such
manifolds, $V$ and $\tilde V$ will be diffeomorphic, we treat this case in \ref{22}.

\item The special family of GT basic classes:  Well known results imply that any surface $V$ arising in this context is stable.

\end{enumerate}

\subsection{Smoothly isotopic surfaces under deformation--Theorem \ref{stability}}\label{21}

With the preparatory work of the previous section, we are now ready to prove Theorem \ref{stability}, which we recall here.

\begin{theorem}\label{41}Let $(M,\omega)$ be a symplectic manifold with $b^+(M)=1$ and $G$ a homological configuration represented by  a curve configuration $V$.
Then $V$ is $\omega$-stable. Moreover, $\tilde V$ can be chosen such that each component of $\tilde V$ is smoothly isotopic to the corresponding (given by $G$) component in $V$.
\end{theorem}

\begin{proof}

The assumption on $V$ implies by Theorem \ref{relcone} that
$\mathcal D_{K_\omega}^V\subset \mathcal C^V_M$.   Since $\tilde
\omega$ is deformation equivalent to $\omega$ and pairs positively
with $A$, it follows that   $[\tilde\omega]\in \mathcal
D_{K_\omega}^V\subset \mathcal C^V_M$, ie. there is a $V-$symplectic
form $\tau$ cohomologous to $\tilde \omega$.

Notice that by Theorem \ref{relcone}  the  $V-$relative symplectic forms are deformation equivalent to
$\omega$. Thus $\tau$ can be assumed to be  deformation
equivalent to $\omega$.
In \cite{M1} it is shown, that when $b^+=1$, any deformation
equivalent cohomologous  symplectic forms are isotopic. It follows
that $\tau$ and $\tilde \omega$ are isotopic. Applying Moser's
Lemma, we obtain a $\tilde \omega-$symplectic curve configuration $\tilde V$
smoothly isotopic to $V$.
\end{proof}

\begin{remark}\hfill
\begin{enumerate}
\item Note that it is not necessary to postulate that the deformation is
through symplectic forms $\omega_t$ such that $[\omega_t]\cdot [V]>0$.

\item We do not claim that any $\tilde \omega-$surface in the class $A$ is
smoothly isotopic to $V$. In fact,  even for a fixed symplectic structure, there are plenty of
non-uniqueness results (see for example \cite{FS}, \cite{TP}, \cite{TP2}, \cite{TP3}).
However, we have the following observation.

\end{enumerate}

\end{remark}

\begin{cor} Suppose $\omega$ and $\tilde \omega$ are two deformation equivalent symplectic forms on a 4-manifold $M$ with $b^+(M)=1$. If $A\in H_2(M,\mathbb Z)$ is a homology class
pairing positively with both $\omega$ and $\tilde \omega$, then
there is a 1-1 correspondence  of smooth isotopy classes of connected $\omega-$ and $\tilde \omega-$symplectic  surfaces in the class $A$.
\end{cor}

As an immediate corollary of this theorem and the light cone lemma we obtain:

\begin{cor}
 In the situation of Theorem \ref{41}, if $A_i\cdot A_i\geq 0$ for all $i\in|G|$, then $\tilde V$ exists for any $\tilde \omega$ deformation
equivalent to $\omega$.

\end{cor}

If $M$ has Kodaira dimension $\kappa(M)=-\infty$, then the deformation class of $\omega$ is determined by the canonical class $K_\omega$
(see \cite{M1}, \cite{KKP}  for rational,
\cite{TJLL2} for irrational ruled manifolds).  This immediately implies the following:

\begin{cor}\label{irrcor}
 In the situation of Theorem \ref{41}, suppose further that $M$ has $\kappa(M)=-\infty$. Then $\tilde V$ exists
for any $G-$positive $\tilde \omega$ with $K_{\tilde \omega}=K_{\omega}$.
In particular, if $A_i\cdot A_i\geq 0$ for all $i\in|G|$, then $\tilde V$ exists
for any $\tilde \omega$ with $K_{\tilde \omega}=K_{\omega}$.

\end{cor}

The methods employed to prove Theorem \ref{stability} are rather
robust and allow some variation.   The following is an example.

\begin{theorem}\label{ADEstab}
Let $(M,\omega)$ be a symplectic manifold with $b^+(M)=1$ and $G$ a
homological configuration represented by a curve configuration
$V$.  Furthermore, let $\Sigma$ be a connected embedded
$\omega$-symplectic curve disjoint from $V$.  Let $\tilde \omega$ be
any symplectic form on $M$ such that the following hold:
\begin{enumerate}
\item  $\tilde \omega$ is deformation equivalent to $\omega$ through forms $\omega_t$ which leave $\Sigma$ symplectic and
\item  $\omega|_\Sigma=\tilde\omega|_\Sigma$.

\end{enumerate}
  Then there exists a curve configuration $\tilde V$ for
  $\tilde \omega$ such that $\tilde V$ is disjoint from $\Sigma$ and $\tilde V$ is smoothly isotopic (with respect to $G$) to $V$.

\end{theorem}

\begin{proof}
Theorem \ref{relcone} provides a symplectic form $\tau$ deformation
equivalent to $\w$, which makes $\Sigma\sqcup V$ $\tau$-symplectic
and with $[\tau]=[\tilde\omega]$. As $\tau$ is deformation
equivalent to $\omega$, we obtain a family $\{\alpha_t\}$ of
symplectic forms from $\tau$ to $\tilde\omega$ which satisfy the
following:
\begin{enumerate}
\item $[\alpha_0]=[\tau]=[\tilde\omega]=[\alpha_1]$;
\item $\Sigma$ is $\alpha_t$-symplectic and
\item $\tau|_\Sigma=\tilde\omega|_\Sigma$

\end{enumerate}
Here (3) is achieved by Moser's method on $\Sigma$.  Then by Thm.
1.2.12, \cite{M3}, there is a family of symplectic forms
$\alpha_{st}$ ($s,t\in[0,1]$) such that $\alpha_{1t}$ is a
cohomologous deformation of $\tau$ to $\tilde \omega$ with
$\alpha_{1t}|_\Sigma=\tau|_\Sigma=\tilde\omega|_\Sigma=\omega|_\Sigma$.
Now apply the Moser Lemma again to obtain a Hamiltonian isotopy that is identity
on $\Sigma$, from which we produce $\tilde V$ as claimed.

\end{proof}

\subsection{Existence of Diffeomorphic surfaces}\label{22}

We briefly remark on the consequences of the results in previous section coupled with actions of diffeomorphism groups; notation introduced
here will be used throughout the rest of the paper.

Recall from \eqref{e:D(M)} that $D(M)$ is the image of the  group of diffeomorphisms Diff(M) in Aut($H_2(M,\mathbb Z)$).
$D(M)$ defines a group action on the set of symplectic canonical classes $\mathcal K$ of $M$.   When $M$ has $b^+(M)=1$, up to sign, $D(M)$ acts transitively on $\mathcal K$ (see \cite{TJLL}, \cite{TJLL2}). For symplectic manifolds with $\kappa(M)=-\infty$ this result can be improved:

\begin{lem}[\cite{TJLL2}, \cite{TJLL}] \label{trans}

If $M$ has Kodaira dimension $\kappa(M)=-\infty$, then the action of $D(M)$ on $\mathcal K$ is transitive.
Furthermore, $D(M)$ is generated by reflections on $(-1-)$ and $(-2)$-smooth spherical classes.  Concretely,
these are $E_i$ and $H-E_i-E_j-E_k$, $i\neq j\neq k\neq i$ for rational manifolds and $E_i$, $F-E_i-E_j$, $i\neq j$
for irrational manifolds.
\end{lem}

This result reduces the problem for symplectic manifolds with $\kappa(M)=-\infty$ to understanding those
classes $A\in H_2(M)$ admitting symplectic representatives for symplectic forms
$\omega$ within a fixed symplectic canonical class $K\in \mathcal K$.

\begin{cor}\label{ex}
 Let $(M,\omega)$ be a symplectic manifold with $\kappa(M)=-\infty$ and $A\in H_2(M,\mathbb Z)$
a homology class admitting an $\omega$-symplectic surface $V$.  Then for every symplectic canonical class $K$ there exists a symplectic form $\tilde \omega$ with $K_{\tilde\omega}=K$ admitting a $\tilde\omega$-symplectic surface diffeomorphic to $V$.

\end{cor}

\begin{proof}
Lemma \ref{trans} provides for an element of $D(M)$ which takes $K$ to $K_\omega$.  This element covers a diffeomorphism; let $\tilde\omega$ be the pull-back of $\omega$ under this map.  The result then follows.

\end{proof}

For a general symplectic manifold, one may consider the following subset of $\Omega(M)$:

\begin{definition} 

Let $\mathcal D(\omega, A)\subset \Omega(M)$ be the set of symplectic forms on $M$ satisfying the following:  For every $\alpha\in \mathcal D(\omega, A)$ there is  a symplectic form  $\beta$ in the Diff$(M)-$orbit of $\alpha$ which has canonical class $K_\beta=K_\omega$ and $[\beta]\cdot A> 0$.

Denote by $\mathcal D^d(\omega,A)\subset \mathcal D(\omega,A)$ the set of classes such that $\beta$ is deformation equivalent to $\omega$.

\end{definition}

Thus $\mathcal D(\omega, A)$ is the orbit under the action of Diff($M)$ of the set $\{\beta\in\Omega(M)\;|\;K_\beta=K_\omega,\;[\beta]\cdot A>0\}$ whereas $\mathcal D^d(\omega,A)$ is the orbit of the path connected component of $\{\beta\in\Omega(M)\;|\;K_\beta=K_\omega,\;[\beta]\cdot A>0\}$ containing $\omega$.  The following results extend the stability results from just the orbit of $\omega$ to these larger sets when
 $\kappa(M)=-\infty$, as a consequence of Lemma \ref{trans}.

\begin{lem} Let $(M,\omega)$ be a symplectic manifold with $\kappa(M)=-\infty$.
Then $\mathcal D(\omega, A)=\mathcal D^d(\omega, A)$ and the restriction of the map
 \[
\begin{array}{ccc}
\Omega(M)&\rightarrow& H^2(M,\mathbb R)\\\omega&\mapsto &K_\omega
\end{array}
\]
to $\mathcal D(\omega, A)$ is onto the set of symplectic canonical
classes.

\end{lem}

Theorem \ref{stability} can be used to obtain the
 following general existence principle for manifolds with $b^+(M)=1$.

\begin{prop}\label{expri} Let $(M,\omega)$ be a symplectic manifold with $b^+(M)=1$ and $A\in H_2(M,\mathbb Z)$
a homology class admitting an $\omega$-symplectic surface $V$.
 Let  $\tilde \omega\in\mathcal D^d(\omega, A)$.  Then there exists an $ \tilde \omega-$symplectic surface $\tilde V$ which is diffeomorphic  to $V$.

\end{prop}

\begin{proof}  Let $\alpha$ be any symplectic form in the Diff$(M)$-orbit of $\tilde\omega$ such that $\alpha$ is deformation equivalent to $\omega$  and $[\alpha]$ pairs positively with $A$.  Then by Theorem \ref{stability} there exists an $\alpha$-symplectic submanifold $V_\alpha$ smoothly isotopic to $V$.  Combining this with the diffeomorphisms taking $\alpha$ to $\tilde\omega$, the result follows.
\end{proof}

Under the additional assumption that $\kappa(M)=-\infty$, two
symplectic forms with a common canonical class are deformation
equivalent \cite{MP94}, hence the above result can be sharpened.

\begin{lem}\label{l:3.11}
Let $(M,\omega)$ be a symplectic manifold  with $\kappa(M)=-\infty$
and $A\in H_2(M,\mathbb Z)$ a homology class admitting an $\omega$-symplectic
surface $V$. Let  $\tilde \omega\in\mathcal D(\omega, A)$.  Then
there exists an $ \tilde \omega-$symplectic surface $\tilde V$ which
is diffeomorphic  to $V$.
%

\end{lem}

Again the light cone lemma allows us to formulate a simple
corollary when $\kappa(M)=-\infty$.

\begin{cor}
 If $A\cdot A \geq 0$, then for any  symplectic form $\tilde \omega$,
there exists an $ \tilde \omega-$symplectic surface $\tilde V$ which is  diffeomorphic to $V$.

\end{cor}

\section{ Lagrangian ADE-configurations}\label{s:ADE}

In this section, as an application of the stability results, we explain how to obtain Lagrangian
ADE-configurations . This
is closely related to the conifold transition, which we will review
in Section \ref{s:conifold}, where a slight refinement of the
deformation type result in \cite{OO05} is shown. A stability result
of Lagrangian ADE-configurations will be explained in
\ref{s:conifold}.  This will eventually lead to a proof of Corollary
\ref{c:ADELag}.

\subsection{Conifold transitions and stability of Lagrangian configurations}\label{s:conifold}

 By definition, in real dimension 4, an
\textit{ADE-configuration of Lagrangian spheres} is a plumbing of
Lagrangian spheres as $A_n$, $n\geq1$; $D_n$, $n\geq4$; or $E_{6,7,8}$
Dynkin diagrams.  These are the \textit{smoothing} of simple singularities of
type $C^2/\Gamma$, where $\Gamma$ is a finite subgroup of $SU(2)$.
On the other hand, one may perform a minimal \textit{resolution} of such
singularities, which incurs a tree-like configuration of
$(-2)$-rational curves, which is of the same diffeomorphism type as
the Lagrangian plumbings in smoothings.  One may replace a
neighborhood of a smoothing by the resolution, or vice versa. Such a
surgery is called a \textit{conifold transition}. See \cite{IR}, \cite{OO05}
for more background on conifold transitions over surfaces.

Note that performing the conifold transition as a symplectic cut-and-paste
surgery has the following features:

\begin{enumerate}[(i)]
   \item If there are $\w$ and $\wt\w$-Lagrangian ADE-configurations $\L=\bigcup_i\{L_i\}$ and $\wt\L=\bigcup_i\{\wt L_i\}$, respectively, then one may choose an appropriate neighborhood of $N$ and $\wt N$ for $\L$ and $\wt\L$ to perform conifold transitions.  Symplectically the conifold transitions remove $N$ and $\wt N$ and replace them by a neighborhood of a symplectic ADE-configuration of (-2)-spheres.  By the Lagrangian neighborhood theorem for configurations (Proposition 7.3 of \cite{IR}) such a symplectic configuration can be chosen isomorphic for both surgeries on $\L$ and $\wt\L$.  In other words, suppose $\w'$ and $\wt\w'$ are the symplectic forms after conifold transitions, where
       $\bigcup_i\{V_i\}$ and $\bigcup_i\{\wt V_i\}$ are the symplectic configurations, then one may choose the surgeries so that $\w'(V_i)=\wt\w'(V_i)$ for all $i$.


   \item Moreover, let $\{[L_i]\}_{i=1}^n$ span a subspace $\cL\subset H_2(M,\R)$.  One may consider its orthogonal complement $\cL^\perp$ under Poincare pairing.  The conifold transition changes the symplectic form, adopting notation from the previous paragraph, in such a way that $\w|_{\cL^\perp}=\w'|_{\cL^\perp}$.  This applies equally well in the other direction of the transition, that is, when changing
        a resolution $\{V_i\}$ to a smoothing.  As a consequence, if $[\w]=[\wt\w]$ and they each admit a symplectic ADE-configuration $\bigcup_i\{V_i\}$ and $\bigcup_i\{\wt V_i\}$ so that $[V_i]=[\wt V_i]$, after changing both configurations to smoothings the new symplectic forms are again cohomologous.

\end{enumerate}

In our situation, we would like to understand the connection between
conifold transition and symplectic deformations.  Symplectically, Ohta and Ono
showed in \cite{OO05} that
any weak/strong symplectic filling of
the link $(\cL,\lambda)$ of an ADE-singularity has a unique
symplectic deformation type, while the deformation is along a family of
weak/strong symplectic fillings.  Here $\lambda$ is the contact form on
the link $\cL$.  We refer readers to \cite{OO05} (or some standard
reference on contact geometry and symplectic fillings, e.g.
\cite{Et04}) for relevant definitions.

Because of the local feature of conifold transitions, it is rather
conceivable that it can be achieved by a compactly supported
symplectic deformation.  In particular this is true for an $A_1$
smoothing in view of symplectic cuts.  Unfortunately we are unable
to prove this: note that this is not a local question, for example,
one cannot obtain a compactly supported symplectic deformation in
$T^*S^2$ so that the zero section becomes symplectic while a fiber
is preserved as a Lagrangian plane due to homological obstructions.
However, we show the following variant of Ohta and Ono's result.  We
emphasize in this result that there is no guarantee that $\w_1$ is the
symplectic form obtained by conifold transition (as a surgery).

\blem\label{l:con=def}

Let $(W,\w)$ be a neighborhood of an ADE-symplectic configuration
$V$ which is a strong filling of $(\cL,\lambda)$.  Then there is a
compactly supported symplectic deformation $\{\w_t\}_{0\le t\le1}$
on $V$, so that $\w_0=\w$ and $V\subset (W, \w_1)$ is a Lagrangian
ADE-plumbing.

The reverse procedure also exists, that is, one has a compactly
supported deformation which transforms smoothings into resolutions.
\elem

\bpf

We only prove the direction from resolution to smoothings, the other
direction is identical.

First perform a conifold transition to $V$ which incurs a symplectic
manifold $(\wt W,\wt\w)$ diffeomorphic to $W$ with a smoothing
configuration $\wt V$.  One identifies $W$ and $\wt W$ smoothly so
that $\wt V$ is identified to $V$, hence the result is
a symplectic form $\w'$ on $W$ so that $V$ is a $\w'$-Lagrangian
configuration and $\w=\w'$ near $\partial W$
since conifold transition only happens in
the interior.

From \cite{OO05}, we have a deformation $\{\Omega_t\}$ which is a
symplectic deformation from $\w$ to $\w'$, where $\Omega_t$ are all
strong fillings of $(\cL,\lambda)$.  By definition, this means in a
collar neighborhood $U$ of $\cL$, with $U\cap V=\emptyset$,
$\Omega_t=d\lambda_t$, and $\lambda_t$ is an extension of
$\lambda$ on $\cL$.  Take $\wt X_t$ so that $i_{\wt
X_t}\Omega_t=\frac{d}{dt}\lambda_t$.  Cut off $\wt X_t$ so that one
obtains $X_t$ which is supported in $U$ and equals $\wt X_t$ in a
smaller $U'\subset U$.  Note that the right hand side vanishes
identically on $\cL$, the flow of $X_t$ is supported away from $\cL$
and creates a family of diffeomorphisms $\varphi_t$ such that
$(\varphi_t)_*\Omega_t=\w$ in $U'$.  Hence
$\{(\varphi_t)_*(\Omega_t)\}_{0\le t\le 1}$ is a compactly supported
deformation of $\w$, while $\w_1=(\varphi_1)_*(\Omega_1)$ contains a
Lagrangian configuration, since $(\varphi_t)_*\Omega_t=\Omega_t$ in the complement of $U$

\epf


With this understood, we may show:

\bthm\label{t:stabilityADE} ADE-configurations of Lagrangian spheres
have the stability property in symplectic manifolds $M$ with
$b^+(M)=1$. Moreover, if $D\subset M$ is a smooth
symplectic divisor, then the stability holds in its complement.\ethm

\bpf We give the proof in the presence of $D$, the case when $D$ is
empty is only easier.  Given a symplectic manifold $(M,\w)$, suppose
it has an ADE-configuration $\L=\bigcup_{i=1}^n\{L_i\}$ consisting
of Lagrangian spheres $L_i$ in $M\backslash D$.  Consider
$\widetilde\w$ deformation equivalent to $\w$ through a compactly
supported deformation family in $M\backslash D$, where
$\wt\w([L_i])=0$.  We would like to show that there exists an
ADE-plumbing of Lagrangian spheres in the complement of $D$.


We proceed as follows.  Apply first Lemma \ref{l:con=def}
to a neighborhood $\cN$ of $\L$, which turns it into a resolution.
This results in a new symplectic form
$\w'$, as well as a symplectic configuration $\bigcup_i V_i$.  Note
that by choosing $\cN$ sufficiently small, one may assume $\w$ and
$\w'$ are $C^0$-close, which is equivalent to saying $\w'|_{\cL}$
being small.

When $\epsilon=||\w-\w'||_{C^0}$ is sufficiently small,
one allows a symplectic deformation from $\wt\w$ to $\wt\w'$, so that
$[\wt\w']|_{\cL}=[\w']|_{\cL}$,
$[\wt\w']|_{\cL^\perp}=[\wt\w]|_{\cL^\perp}$ and $\epsilon>||\wt\w-\wt\w'||_{C^0}$.
This can be achieved by packing-blowup correspondence \cite{MP94} for the following reason.
Both $\cL$ and $\cL^\perp$ are spanned by subcollections in $\{H,E_1,\dots,E_n\}$, while one has the freedom
to adjust the symplectic areas of each:
sizes of ball-packings corresponding to symplectic areas of $E_i$ which can be adjusted slightly
by the continuity of packing, while the area of $H$ can be adjusted by a global rescaling.  Also, note that
when $\epsilon$ is sufficiently small, $D$ is preserved as a symplectic divisor.

 Now apply the Stability Theorem \ref{ADEstab} for the symplectic configuration $\bigcup_i V_i$
 from $\w'$ to $\wt\w'$ and divisor $D$ as $\Sigma$.  This implies the existence of a symplectic configuration with
 respect to $\wt\w'$ in the complement of $D$.  One can then use Lemma \ref{l:con=def} in a reverse direction on this
 configuration to obtain a smoothing (Lagrangian configuration of spheres) in $M$ with a certain symplectic form $\wt\w''$.
 Note that $\wt\w$ and $\wt\w''$ are deformation equivalent by concatenating the symplectic
 deformation from $\wt\w$ to $\wt\w'$, and they are cohomologous by (i) and (ii) (because this reverse
 conifold transition only ``erases" the symplectic form on $\cL$ and leaves $\cL^\perp$ invariant).

 Applying Thm. 1.2.12, \cite{M3}, one may deform such a symplectic
 deformation to an isotopy of symplectic forms $\Omega_t$, $\Omega_0=\wt\w$ and
 $\Omega_1=\wt\w''$ by symplectic inflations while preserving $D$ as a symplectic divisor.
 Along this isotopy of symplectic forms, $D$ has constant symplectic area.
 Therefore, one may choose a diffeomorphism $\tau_t$ supported near $D$,
 so that $\tau_t^*(\Omega_t)$ is constant on $D$.
 Now Moser's method on $\tau_t^*(\Omega_t)$ yields an isotopy
 $\phi_t$ which is identity restricted to $D$, where $(\phi\circ\tau_1)^*(\Omega_1)=\wt\w$.
 Then the $\phi_1$-image of the constructed $\wt\w''$-Lagrangian configuration is as desired in the complement of $D$.

\epf

\subsection{Existence}

In \cite{LW}, the second and third authors derived a necessary and
sufficient condition for $A\in H_2(\CP^2\# k\overline{\CP}^2,\Z)$ to
admit a Lagrangian spherical representative: this holds if and only
if $A$ is $D(M)$-equivalent to $E_1-E_2$ or $H-E_1-E_2-E_3$ and
$[\omega]\cdot A=0$.  With the stability result above, we may
improve the existence part into existence of ADE-smoothings.
\cite{OO05} explained how to compactify an ADE-type smoothing into a
rational manifold of diffeomorphism type
$\CP^2\#(n+3)\overline{\CP^2}$.  After compactification, a
symplectic neighborhood of the Lagrangian configuration can be
recovered by removing a set of smooth symplectic divisors from the
rational surface.  The homology classes of these divisors are listed
as follows:

\begin{itemize}
  \item $A_n: H, H-E_1-\dots-E_{n+1}$;
  \item $D_n: E_1, E_2-E_1, H-E_1-E_2-E_3, 2H-E_1-E_2-E_4-\dots-E_{n+3}$;
  \item $E_n: E_1, E_2-E_1, E_3-E_2-E_1, 3H-2E_3-E_4-\dots-E_9(-E_{10}-E_{11})$.
\end{itemize}

Note that for the case of $A_n$ we have used a particularly simple
set of divisors slightly different from that in \cite{OO05}, where
we have $\CP^2\#(n+1)\overline{\CP^2}$ as the ambient rational
surface.  The corresponding homology classes of the Lagrangian
ADE-configurations are given as follows:

$$A_n: 
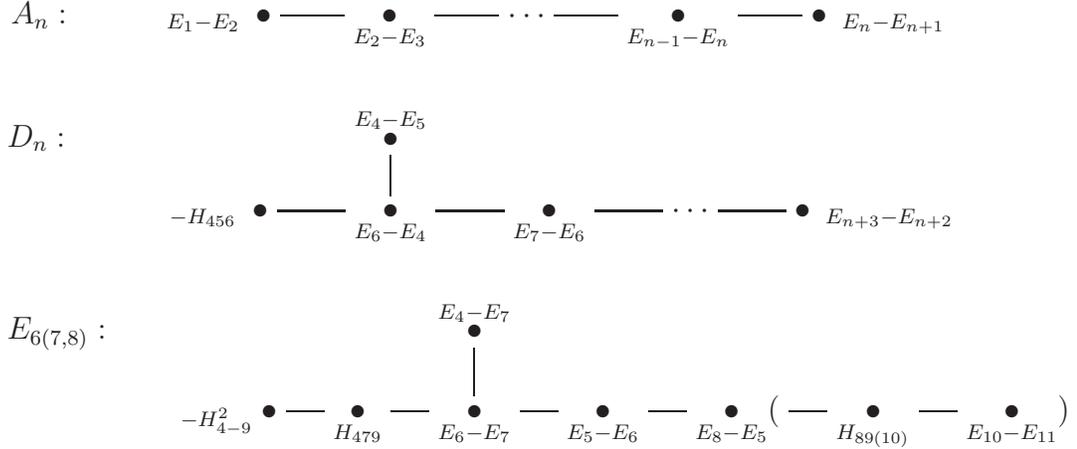
\xymatrix{
&{}_{E_1-E_2}\hskip 2mm\bullet\ar@{-}[r]
&{\underset{E_2-E_3}{\bullet}}\ar@{-}[r] &\cdots\ar@{-}[r]
&{\underset{E_{n-1}-E_{n}}{\bullet}}\ar@{-}[r] &\bullet_{\hskip2mm
E_{n}-E_{n+1}}}$$

$$D_n: \xymatrix@C=.9cm@R=0.2cm{
& &\overset{E_4-E_5}{\bullet}\ar@{-}[d] \\
&{}_{-H_{456}}\hskip 2mm\bullet\ar@{-}[r]
&{\underset{E_6-E_4}{\bullet}}\ar@{-}[r]
&\underset{E_7-E_6}{\bullet}\ar@{-}[r]&\cdots\ar@{-}[r]
&\bullet_{\hskip2mm E_{n+3}-E_{n+2}}}$$

\begin{equation*}
E_{6(7,8)}: \xymatrix@C=.5cm@R=0.2cm{
& & &\overset{E_4-E_7}{\bullet}\ar@{-}[d] \\
&{}_{-H^2_{4-9}}\hskip 1mm\bullet\ar@{-}[r]
&{\underset{H_{479}}{\bullet}}\ar@{-}[r]
&\underset{E_6-E_7}{\bullet}\ar@{-}[r]&\underset{E_5-E_6}{\bullet}\ar@{-}[r]
&\underset{E_8-E_5}{\bullet}(\ar@{-}[r]&\underset{H_{89(10)}}{\bullet}\ar@{-}[r]&\underset{E_{10}-E_{11}}{\bullet})}
\end{equation*}

Here $H_{ijk}$ and $H^2_{4-9}$ are shorthand for $H-E_i-E_j-E_k$ and $2H-E_4-\dots-E_9$, respectively.
Motivated by these explicit identifications, we define:

\begin{definition}
  A set of homology classes $\{l_i\}_{i=1}^n\subset H_2(\CP^2\# n\overline{\CP}^2,\Z)$ is a
  {\textbf{homological Lagrangian ADE-configuration}}
  if there is a $D(M)$-equivalence $\tau$ on $H_2(\CP^2\# k\overline{\CP}^2)$ so that
  $\{\tau(l_i)\}_{i=1}^n$ are of the form specified above and $\w(l_i)=0$.
\end{definition}

We are now ready to prove Corollary \ref{c:ADELag}, which we recall below.

\bcor\label{c:existenceADE} In rational or ruled 4-manifolds, any homological Lagrangian
ADE-configuration $\{l_i\}_{i=1}^n$ admits a Lagrangian
ADE-configuration representative. In the case of
$A_n$-configurations, one may require the configuration lie in
$M\backslash D$, where $D$ is a symplectic divisor disjoint from a
set of embedded symplectic representatives of $\{E_i\}_{i=1}^{n+1}$.

\ecor

\bpf Notice first that any reflection along a $-2$-sphere  is the homological action of
a diffeomorphism, therefore we may assume the homological
configuration is precisely of the form specified in \cite{OO05}.
Choose an arbitrary symplectic form $\w'$ so that our designated
classes admit a symplectic ADE-configuration representative (see \cite{OO05} for an example).  Then one may
obtain a Lagrangian ADE-configuration by a conifold transition, by
changing the symplectic form to some $\w''$.  Note that $\w$ and
$\w''$ are symplectic deformation equivalent (as is the case for any
symplectic form in rational manifolds with the same canonical
class), our existence result is an immediate consequence of Theorem
\ref{t:stabilityADE}.

For the $A_n$ case, we refine our embedding of the Lagrangian
configuration as follows.  Blow down all $E_i$, $i\le n+1$ and
shrink the resulting balls to a very small equal size, then isotope
them into a Darboux neighborhood.  Upon blowing back up these small
balls one obtains a symplectic form on $M$ with an open set
symplectomorphic to $B^4\#(n+1)\overline{\CP^2}$, where all
exceptional spheres have the same symplectic area.  This open
set contains a Lagrangian $A_n$-configuration, see for example, the construction
 in Section 2 of \cite{W13}.  While the isotopy above can be
chosen disjoint from $D$, the deformation is supported disjoint from
$D$, as well.  Therefore, one may apply the stability result in Theorem
\ref{t:stabilityADE} above.

\epf

\bpf[Proof of Corollary \ref{c:An-b1} and \ref{c:An}]  Note that
we may reduce Corollary \ref{c:An-b1} to the case of \ref{c:An}, that is,
 when the packing is supported away from the isotropic skeleton.
 To see this, rescale the symplectic form on $M$ so that $\w(E_i)$ are rational numbers.
 Then choose a deformation so that the symplectic form of the minimal model of $M$
 has rational period, this can be done due to the openness of the non-degeneracy condition.
 One then shrinks all embedded balls corresponding to
 each exceptional sphere (including those not listed as $E_i$, but consisting basis elements in $H_2(M)$) to a very small
 volume and then move them away from an isotropic skeleton of the minimal model
 of $M$.  The blow-up along such small balls thus gives a form $\w'$ which is deformation equivalent to
 the original symplectic form.  If one has an $A_n$-Lagrangian configuration
 for $\w'$, then the stability for manifolds with $b^+=1$ in
 Theorem \ref{t:stabilityADE} concludes Corollary \ref{c:An-b1}.


Therefore, it suffices to find a Lagrangian $A_n$-configuration in
the complement of the isotropic skeleton when the minimal model of
$M$ has rational period (case of Corollary \ref{c:An}).  Biran
\cite[Theorem 1.A]{Bi01} showed that this complement is
symplectomorphic to a standard symplectic disk bundle $E$ modelled
on the normal bundle of a Donaldson hypersurface.  One may then
compactify this disk bundle $E$ into a symplectic ruled surface $E'$
by slightly deforming the symplectic form and adding a symplectic
divisor at infinity (equivalently, do a symplectic cut near the
boundary).  Upon blowing up, one may apply the existence result
Corollary \ref{c:ADELag} with $D$ as the added infinity divisor. The
corollary is thus concluded by embedding the complement of $D$ back
into the complement of the isotropic skeleton.

\epf

\section{Spheres in Rational Manifolds}\label{s-4}

In this section, we prove the classification result Theorem \ref{t:spheres}, using the following strategy:
\begin{itemize}
\item we first provide a classification of homology classes that satisfy the imposed constraints, see Section \ref{s:Hom-4};
\item secondly, we show that all classes obtained in this way are symplectically representable by a connected $\omega$-symplectic
sphere for some symplectic structure $\omega$, using the so-called \textit{tilted transport} of Section \ref{s:Con-4};
\item finally, we apply the results in Section \ref{nodal} to extend the result
to all classes satisfying conditions (1)-(3) in Spec. \ref{spec:isphere}.
\end{itemize}

In Section \ref{s:-1-2-3} we also include a complete account for symplectic
$-1,-2,-3$-spheres for completeness; these results mostly follow
from earlier work, see \cite{L1,TJLL,LW}.

\subsection{Homology classes of smooth $-4$ spheres}\label{s:Hom-4}

Consider a class $A\in H_2(M,\mathbb Z)$ for $M=\mathbb CP^2\#k\overline{\mathbb CP^2}$.  In the standard basis we write $A=aH-\sum_{i=1}^k b_iE_i$.  Such a class is called {\it reduced} if\begin{itemize}
\item $b_1\ge b_2\ge ...\ge b_k\ge 0$ and
\item $a\ge b_1+b_2+b_3$.
\end{itemize}

The following lemma gives a complete list of non-reduced classes for smooth $-4$ spheres.

\begin{lemma}\label{l:hom}
Let $M=\mathbb CP^2\#k\overline{\mathbb CP^2}$ with $k\ge 1$ and $A\in H_2(M,\mathbb Z)$.  Assume that $A\cdot A=-4$.  Then, up to $D(M)$-equivalence, $A$ is a reduced class (when $k\ge 3$),  or  one of the following: $$ -H+2E_1-E_2, \quad 2E_1, \quad 2(H-E_1-E_2),  \quad H-E_1-..-E_5.$$

\end{lemma}

\begin{proof}
For $k=1$, $(aH-bE_1)^2=a^2-b^2=-4$ implies that only $\pm 2E_1$ is possible.

For $k=2$, Lemma 1, \cite{LL4} reduces the problem to classes with $2a\le b_1+b_2$.
Thus $3b_1^2-2b_1b_2+3b_2^2\le 16$ and the only possible classes are $D(M)$-equivalent
to $2E_1$, $2(H-E_1-E_2)$ and $-H+2E_1-E_2$.

For $k\ge 3$, as in Lemma 3.4, \cite{LL2}, it can be shown using reflections along $-2$-spheres $H-E_i-E_j-E_k$ that either $A$ is
$D(M)$-equivalent to a reduced class or to one that satisfies
\[
b_1^2+b_2^2+b_3^2-4\le a^2\le \frac{3}{4}(b_1^2+b_2^2+b_3^2).
\]
In addition to this inequality, $a^2-\sum b_i^2=-4$ and $b_1\ge b_2\ge ...\ge b_k\ge 0$.  The solutions to this system, written in short as $(a,b_1,b_2,..,)$, are:
\[
(0,1,1,1,1,0..), (1,1,1,1,1,1,0,..), (0,2,0,..), (1,2,1,0,..),
\]
\[
(2,2,1,1,1,1,0,..), (2,2,2,0,..), (3,2,2,2,1,0,..), (3,3,2,0,..).
\]
Under the $D(M)$-action, $(0,2,0,..)\leftrightarrow (2,2,2,0,..)$, $(1,1,1,1,1,1,0,..)$ is in a class of its own when $k=5$. The other classes are all equivalent, and when $k\ge 6$, $(1,1,1,1,1,1,0,..)$ is included as well.

\end{proof}

\subsubsection{Symplectic Genus}

In order to address the reduced classes in Lemma \ref{l:hom}, we first
briefly describe a general obstruction to the existence of smooth /symplectic surfaces in a symplectic manifold.

Clearly, for a class $A\in H_2(M,\mathbb Z)$ of a symplectic manifold $(M,\omega)$ to be represented by a symplectic surface, there must exist $\alpha\in\mathcal C_M$ with $\alpha\cdot A>0$.  Let $\mathcal K$ denote the set of symplectic canonical classes.  Consider the following set:
\[
\mathcal K_A=\{K\in\mathcal K\;|\;\exists \alpha\in\mathcal C_M: K_\alpha=K, \alpha\cdot A>0\}.
\]
To each $K\in \mathcal K_A$, define $\eta_K(A)=\frac{1}{2}(K\cdot A+A\cdot A)+1$.  Finally, define the {\it symplectic genus} to be
\[
\eta(A)=\max_{K\in\mathcal  K_A}\eta_K(A).
\]
Note that there is no guarantee that $\eta(A)\ge 0$.  If $K\in\mathcal  K_A$ is
some symplectic canonical class such that $\eta(A)=\eta_K(A)$, we obtain the inequality
\[
\tilde K\cdot A\le K\cdot A
\]
for any $\tilde K\in\mathcal  K_A$.  Moreover, Lemma 3.2, \cite{LL2} shows that $\eta(A)$ has the following properties:\begin{enumerate}
\item The symplectic genus $\eta(A)$ is no larger than the minimal genus of $A$.  Moreover, if $A$ is represented by a connected symplectic surface, then the minimal genus and the symplectic genus coincide.
\item The symplectic genus is invariant under the action of Diff(M).
\end{enumerate}
Notice that the first condition ensures that the symplectic genus is well-defined as well as providing an obstruction to the existence of a smooth / symplectic curve.

For reduced classes $A$ in non-minimal rational or ruled manifolds, Lemma 3.4, \cite{LL2}, proves that $K_{st}\in \mathcal K_A$.  Thus we obtain the following:

\begin{lemma}\label{l:Kbound}
Let $M$ be a non-minimal rational or ruled manifold and $A\in H_2(M,\mathbb Z)$. Assume that $A$ is reduced and $A$ can be represented by a smooth sphere. Then  $K_{st}\cdot A\le -2-A^2$.
\end{lemma}


\noindent\textbf{Example.} The class $e=(11,6,6,6,1,..,1)\in H_2(\mathbb CP^2\#18\overline{\mathbb CP^2},\mathbb Z)$ satisfies $e\cdot e=-4$ and the
adjunction equality for an embedded sphere for some $K\in\mathcal K$.  However, it is $D(M)$-equivalent to the reduced class $e_r=(4,1,..,1)$, which has symplectic genus $1$, hence
cannot be represented by a smooth embedded sphere.  It then follows from stability that the same must hold for $e$.  Note also that $e$ is not
\textit{Cremona equivalent} (reflections with respect to only (-2)-spherical classes) to a reduced class, this consideration distinguishes
this from the approach of \cite{LW}.

\subsubsection{Reduced $-4$ classes} \label{s:-4red}

We will now begin a study of the possible reduced classes.  Let $M=\mathbb CP^2\#k\overline{\mathbb CP^2}$ with $k\ge 1$ and $A\in H_2(M,\mathbb Z)$. Assume that $A$ is reduced,  $A\cdot A=-4$ and $A$ can be represented by a smooth sphere. Thus Lemma \ref{l:Kbound}  implies that
 $K_{st}\cdot A\le 2$.  Concretely, for some $d\in\mathbb Z$, $d\le 2$ and $\tau\ge 0$, the array of coefficients $(a,b_1,b_2,...,b_k)\in \mathbb \Z^k$ of such a class $A$ solves:

\begin{align}
\label{e:3a}&3a=\sum_{i=1}^k b_i-d\\
\label{e:a2} &a^2=\sum_{i=1}^k b_i^2-4\\
\label{e:aineq}&a\ge b_1+b_2+b_3\\
\label{e:bconstraints}&b_i\ge b_{i+1}\ge \tau,\hskip2mm i=1,\dots,k-1.
\end{align}

Notice that for a class to have negative self-intersection and be reduced, we must have $k\ge 10$.
The role of $\tau$ will become transparent in the proof, we will only consider cases with
$\tau\in\{1,2,3\}$ (the resulting equations are in fact not exclusive).
In summary, the standing assumptions for the set of equations \eqref{e:3a}-\eqref{e:bconstraints} are:

\begin{equation}\label{e:standingAss}d\le2, d\in\Z,\hskip 2mm k\ge10,\mbox{  and  } \tau\in\{1,2,3\}.\end{equation}

The goal is now to show that $d=K_{st}\cdot A=2$ is the only
possibility, even when we relax the condition to allow
$(a,b_1,b_2,...,b_k)\in \mathbb \R^{k+1}$.  For this we need to
describe a rearrangment operation which will allow us to rule out these cases.

\begin{lemma}\label{l:reduce}
Assume a solution $(a,b_1,\dots,b_{k})\in \mathbb \R^{k+1}$ to \eqref{e:3a}-\eqref{e:bconstraints} exists when either
\begin{itemize}
\item $\tau=1$ and $k\ge 11$ or
\item $\tau\in\{2,3\}$ and $k=10$ or
\item $\tau=1$, $d<2$ and $k=10$.
\end{itemize}
  Then there exists a solution $(a,b_1',...,b_{k}')\in \mathbb \R^{k+1}$ to \eqref{e:3a}-\eqref{e:bconstraints} which further satisfies:

\begin{align}
 \label{e:aeq}a&= b'_1+b'_2+b'_3 \\
 b'_2&=\dots=b'_{4+r-1}=B \\
 b'_{4+r+1}&=\dots=b'_{k}=\tau \\
 \label{e:b'}\tau&\le b'_{4+r}=b'\leq B
\end{align}
where $0\le r\le k-4$.
\end{lemma}

\bpf
We first describe a rearrangement operation on a solution to \eqref{e:3a}-\eqref{e:bconstraints} which changes the $b_i$ while leaving $a$ unchanged and preserving all but \eqref{e:a2}.  Suppose ${\bf s}=(a,b_1,\dots,b_{k})$ is a solution to \eqref{e:3a}-\eqref{e:bconstraints}.  Assume $b_i>b_j>0$.  Then for $c\in\mathbb R^+$, replace $(b_i, b_j)$ by $(b_i+c,b_j-c)$.  This operation clearly leaves \eqref{e:3a} unchanged and, by properly choosing $c$, preserves \eqref{e:aineq} and \eqref{e:bconstraints}.  We will always assume that $c$ has been chosen in this manner. After such an operation, $b_i^2+b_j^2$ will increase at least by $2c^2$.

Now apply this operation repeatedly choosing  $b_i\in\{b_1, b_2, b_3\}$  and $b_j\in\{b_4,..,b_{k}\}$.  One arrives at one of the following scenarios:\begin{itemize}
\item $a=\tilde b_1+\tilde b_2+\tilde b_3$ or
\item $a>\tilde b_1+\tilde b_2+\tilde b_3$ and $\tilde b_4=\dots=\tilde b_{k}=\tau$.
\end{itemize}

In the first case, one then further rearranges $\tilde b_1$ with $\tilde b_2$ until $\tilde b_2=\tilde b_3$.  Then rearrange $\tilde b_4$ with the rest until $\tilde b_4=\tilde b_3$ or $\tilde b_5=\cdots=\tilde b_{k}=\tau$.
If $\tilde b_4=\tilde b_3$, do further rearrangements so that $\tilde b_5=\tilde b_4$, etc.  In the second case, rearrange $\tilde b_2$ and $\tilde b_3$ with $\tilde b_1$ until $\tilde b_2=\cdots=\tilde b_{k}=\tau$.

The end result is a new sequence

$${\bf s'}=(a,b'_1,\dots,b'_{k})$$

that satisfies \eqref{e:3a}, \eqref{e:aeq}-\eqref{e:b'} for some $0\le r\le k-4$ as well as one of the following:
\begin{enumerate}
\item $a= b'_1+ b'_2+ b'_3$ or
\item $a> b'_1+2\tau$ and $ b'_2=\dots= b'_{k}=\tau$.
\end{enumerate}
Notice that ${\bf s'}$ will not necessarily satisfy \eqref{e:a2}, instead one has

\beq\label{e:E2'I} a^2\le \sum b'^2_i-4 \eeq

{\bf Case 1:} Consider $a=b'_1+ b'_2+ b'_3$.  The function $F({\bf s})=a^2-b_1^2-\dots-b_{k}^2+4$ satisfies $F({\bf s'}) \leq 0$.

Let $b_1''=\frac{k-7}{2}\tau-\frac{d}{2}$, $a''=b_1''+2\tau$ and  ${\bf s''}=(a'', b_1'', \tau, .., \tau)$.   Then in all cases to be considered  we have
\[
b_1''\ge \tau\mbox{  and  }F({\bf s''})\ge 0.
\]
Thus ${\bf s''}$ satisfies \eqref{e:3a}, \eqref{e:aeq}-\eqref{e:b'} just as ${\bf s'}$ does.

Therefore, the line segment between ${\bf s''}$ and ${\bf s'}$ in $\R^{k+1}$
must contain a solution to $F({\bf s})=0$. Moreover, such a solution must satisfy \eqref{e:3a} and \eqref{e:aeq}-\eqref{e:b'}
since all these conditions are convex and the endpoints of the chosen segment satisfy all these restrictions.

{\bf Case 2:} Consider the situation that we obtain a solution with $a> b'_1+2\tau$ and $ b'_2=\dots= b'_{k}=\tau$.  By solving  $3a=b'_1+(k-1)\tau-d$ for $(k-1)\tau$ and substituting this into \eqref{e:E2'I} and making use of $a>b'_1+2\tau$, we obtain
\begin{equation}\label{case2}
\left(b'_1+\frac{\tau}{2}\right)^2<\left(b'_1-\frac{\tau}{2}\right)^2+2\tau^2+d\tau-4
\end{equation}

When $\tau=1$ (independent of $k$ in fact), \eqref{case2} admits no solution $b'_1\ge 1$ when $d\le 2$.

Assume now that $\ \tau\in\{2,3\}$ and $  k=10$.  Then  \eqref{case2} simplifies to
\[
b_1\tau<\tau^2+\frac{d}{2}\tau-2.
\]
Assume that $\tau=2$.  Then $2\le b_1<1+\frac{d}{2}\le 2$, which is a contradiction.\\
Assume that $\tau=3$.  Then $3\le b_1<3+\frac{d}{2}-\frac{2}{3}$ which has no solution when $d\le 1$.
For $d=2$, consider again \eqref{e:3a}.  Solving for $b_1$ under the assumption $a> b'_1+2\tau$, one obtains $3\le b_1<\frac{5}{2}$, also
a contradiction.  Hence Case 2 never shows up and the proof is completed.

\epf

\begin{remark} Consider $d\le 1$ and replace $4$ by $3$ in \eqref{e:a2} to
consider classes with $A\cdot A=-3$.  Then  Lemma \ref{l:reduce}
continues to hold for $\tau=1$ and $k\ge 10$.  This can easily be
seen in Case 2, where in \eqref{case2} the final term changes to 3.
Moreover, in Case 1 the same point ${\bf s''}$ can be used.

It should be noted that in Case 1, the case $\tau=1$, $d=2$ and
$k=10$ does not work. It can be shown that in this setting the
procedure will not terminate with a solution as described.  The
reason for this becomes clear when one considers Lemmata \ref{l:b=1}
and \ref{l:b=2}.
\end{remark}

Making use of this process, we now begin to rule out certain reduced classes.

\begin{prop}\label{t:-4smooth}Let $M=\mathbb CP^2\#k\overline{\mathbb CP^2}$.  Then there exists no reduced class $A\in H_2(M,\mathbb Z)$ with $A\cdot A=-4$ and that $\min\{b_i\}\ge1$ in the following cases:
\begin{enumerate}
\item $k\ge 11$ and $K_{st}\cdot A\le 2$;
\item $k=10$ and $K_{st}\cdot A< 2$;
\item $k=10$, $K_{st}\cdot A = 2$ and $\min\{b_i\}\ge 3$.
\end{enumerate}

\end{prop}

Before we pass to the proof, let us briefly consider the ramifications of this result.
Recalling Lemma \ref{l:hom}, this result shows that when $k\ge 11$, we have no reduced classes with $A\cdot A=-4$ which can be represented by a smooth sphere.  Moreover, according to this result, when $k=10$,  if $A$ is to be represented by a smooth sphere, then $A$ must satisfy $K_{st}\cdot A=2$ and $b_{10}\in\{1,2\}$.  The latter cases will be considered after the proof.

\begin{proof}  We will proceed to show that there exists no solution to \eqref{e:3a}, \eqref{e:a2} and \eqref{e:aeq}-\eqref{e:b'} under the conditions given in the theorem.  The three cases correspond to
\begin{enumerate}
\item $k\ge 11$, $d\le 2$ and $\tau=1$;
\item $k=10$, $d< 2$ and $\tau=1$;
\item $k=10$, $d = 2$ and $\tau= 3$.
\end{enumerate}

To simplify notation, drop all ' in \eqref{e:aeq}-\eqref{e:b'}.

Consider first the case $\tau=1$ and $k\ge 11$.  Then as $a=b_1+2B$, we obtain
\begin{equation}\label{e:2b1}
2b_1=(r-4)B+b+k-4-r-d
\end{equation}
from \eqref{e:3a} and using this in \eqref{e:a2} it follows that
\begin{equation}\label{e:11}
0=(r-6)B^2+2Bb-b^2+(k-4-r-d)(2B-1)+4-d.
\end{equation}
Notice that $2Bb-b^2$ and $4-d$ are strictly positive.

{ $\bf 6\le r\le k-6$}: In this case, $(r-6)B^2$ and $(k-4-r-d)(2B-1)$ are non-negative, hence no solution exists.  As $k\ge 11$, we need to consider $k=11$ and $r=6$ separately:  \eqref{e:11} reduces to
\[
0=2Bb-b^2-2B+(2-d)2B+3=(B-1)^2-(B-b)^2+(2-d)2B+2
\]
from which it can be seen that no solution exists if $b\ge 1$.

{$\bf r=k-5>6$}:  When $r=k-5>6$, then \eqref{e:11} reduces to
\[
0=(k-12)B^2+2Bb-b^2+(B-1)^2+(2-d)2B+2
\]
which again has no solution.

{$\bf r=k-4$:}  \eqref{e:11} can be rewritten as
\[
0=(k-11)B^2+2Bb-b^2+(B-2)^2+3
\]
which admits no solution as $B\ge b\ge 1$.

For the following cases, determine
\[
k-4-r-d=2b_1-(r-4)B-b
\]
and insert into \eqref{e:11} to obtain
\begin{equation}\label{e:simp}
0=(2-r)B^2+4Bb_1-b^2-2b_1+(r-4)B+b+4-d
\end{equation}
Note that $2b_1-(r-4)B-b\ge 0$, and thus if $d=2$ and $r=5$ we must have $k\ge 11$.  This is the cause for the restriction to $d<2$ in the case $k=10$ and $\tau=1$.  Therefore, all of the following arguments continue to hold when $k=10$, $d<2$ and $\tau=1$.

{$\bf r=0$}:  In this case \eqref{e:simp} becomes
\[
0=2B^2+4b_1B-b^2-2b_1-4B+b+4-d=
\]
\[
=B^2-b^2+2Bb_1-2b_1+(B-2)^2+2Bb_1-d+b
\]
which admits no solution.

{$\bf r=1$}:  As before,  \eqref{e:simp} becomes
\[
0=B^2-b^2+2Bb_1-2b_1+2Bb_1-3B+b+2+2-d
\]
where $2Bb_1-3B+b+2\ge \frac{1}{2}$.  Hence no solution exists.

{$\bf r=2:$}  Again \eqref{e:simp} becomes
\[
0=4Bb_1-b^2-2b_1-2B+b+4-d
\]
which can be rewritten to show that no solution exists here either.

{$\bf r=3,4:$} Write $b_1=B+\alpha$ and insert into \eqref{e:simp}.  Then again it can be shown that no solution exists.

{$\bf r=5:$} Again write $b_1=B+\alpha$ and insert into \eqref{e:simp}.  When $\alpha\ge 2$ it easily follows that  there exists no solution.  Otherwise $|b_1-B|<\frac{1}{2}$ and using this in \eqref{e:simp} to succesively estimate the differences of the $B$ and $b_1$ terms it can again be shown that no solution exists.

This completes the case with $\tau=1$ and $k\ge 11$.  As noted before, the cases with $0\le r\le 5$, $k=10$, $d<2$ and $\tau=1$ have also been completed.  It remains to consider $r=6$ in this setting.

{$\bf r=6$, $\bf d<2$ and $\bf k=10$:}  \eqref{e:11} reduces to
\[
0=2Bb-b^2-2dB+4=(B-1)^2-(B-b)^2+2B(1-d)+3
\]
which admits no solution when $d\le 1$.

We now turn to $\tau=3$, $k=10$ and $d=2$.   Rewriting \eqref{e:3a}, \eqref{e:a2} and \eqref{e:aeq}-\eqref{e:b'} it follows that

\beq\label{e:checkKey} [(r-5)B^2+(32-6r)B+9r-50]-(B-b')^2=0;\eeq
\beq\label{e:checkIneq}2b_1=(r-4)B+b'+3(6-r)-2.\eeq

One can then check the compatibility of these equations, case-by-case, for $r$ from $0$ to $6$.  For $r=6$ one easily shows explicitly $B=\frac{b'+2}{2}<b'$ from the first equation.  For $r=5$, from \eqref{e:checkIneq}, $2b_1=B+b'+1$ so $B-b'\le 1$. This again contradicts \eqref{e:checkKey}.

For $r\le4$, from \eqref{e:checkIneq} one deduces $2b_1\le(r-3)b'+16-3r$, implying $b_1\le4$ when $r=4$ and $b_1\le\frac{7}{2}$ when $r\le3$.  The minimum of the $B$-quadratic expression in \eqref{e:checkKey} is taken at $B=3$ when $1\le r\le4$ and $B=\frac{7}{2}$ for $r=0$.
Also $-(B-b')^2\ge -1$ for $r=4$ and $\ge-\frac{1}{4}$ for $r\le3$. Each case will imply a positive minimum in \eqref{e:checkKey}, which concludes our proof.

\end{proof}

As noted before, when $k=10$ this result implies that  $A$ must satisfy $K_{st}\cdot A=2$ and $b_{10}\in\{1,2\}$.  We now show that if $b_{10}=1$ we obtain no solutions either.


\begin{lemma}\label{l:b=1}Let $M=\mathbb CP^2\#10\overline{\mathbb CP^2}$.  Then there exists no reduced class $A\in H_2(M,\mathbb Z)$ with $A\cdot A=-4$, $K_{st}\cdot A\le 2$ and $b_{10}=1$.

\end{lemma}

\begin{proof}
Assume such a class exists.  If $A=(a,b_1,..,b_9,1)$, then the class $(a,b_1,..,b_9)$ in $H_2(\mathbb CP^2\#9\overline{\mathbb CP^2},\mathbb Z)$ is reduced and has square $-3$.  However, an easy computation shows every reduced class in $\mathbb CP^2\#9\overline{\mathbb CP^2}$ has non-negative square.  Hence no such class exists.
\end{proof}

\begin{lemma}\label{l:b=2}
Assume that a reduced class $A\in H_2(\mathbb CP^2\#10\overline{\mathbb CP^2},\mathbb Z)$  satisfies $A\cdot A=-4$ and $K_{st}\cdot A= 2$ with the additional restriction that $b_{10}=2$.  Then $A=-a(-3H+\sum_{i=1}^9 E_i)-2E_{10}$ for some $a\in \mathbb N$ and $a\ge 2$.

\end{lemma}

\begin{proof}
Then the tuple $(a,b_1,..,b_9)\in \mathbb Z^{10}$ satsifies
\begin{align}
3a=\sum b_i\\
 a^2=\sum b_i^2\\
a\ge b_1+b_2+b_3\\
b_i\ge b_{i+1}\ge 2
\end{align}
hence defines a reduced class  $A_9\in H_2(\mathbb CP^2\#9\overline{\mathbb CP^2},\mathbb Z)$.  Using the formula given for $f_{A_9}$ (to determine the minimal genus) in \cite{BL1}, we obtain $f_{A_9}= 1$.   Theorem 2, \cite{BL1}, thus implies that  $A_9=-a(-3H+\sum_{i=1}^9 E_i)$ for some $a\in \mathbb N$ and $a\ge 2$.  Therefore $A=A_9-2E_{10}$ as claimed.

\end{proof}

\subsubsection{Classification}
Together the results in this section lead to the following Theorem,
which completes the smooth classification of Theorem
\ref{t:spheres}, as well as implies the exclusiveness part of the
symplectic classification of Theorem \ref{t:spheres}:

\begin{theorem}\label{t:-4dmequiv}Let $M=\mathbb CP^2\#k\overline{\mathbb CP^2}$ with
$k\ge 1$ be a symplectic rational surface, and $A\in H_2(M,\mathbb
Z)$ with $A\cdot A=-4$. Then $A$ is represented by a smooth sphere
if and only if $A$ is D(M)-equivalent to one class in the following
list
\begin{enumerate}

\item $-H+2E_1-E_2$
\item $H-E_1-..-E_5$
\item $-a(-3H+\sum_{i=1}^9 E_i)-2E_{10}$ for some $a\in \mathbb N$ and $a\ge 2$
\item $2E_1$
\item $2(H-E_1-E_2)$

\end{enumerate}

Moreover,  for $A$ in (1), (2), (3),    $K_{st}\cdot A=2$, and there
is a symplectic form  $\omega$  with $K_\omega=K_{st}$ such that
$A\cdot [\omega]>0$; there is no symplectic form $\tau$ with canonical class $K_\tau$ satisfying
$[\tau]\cdot A>0$ and $K_\tau\cdot A=2$ for classes $A$ of the form in
(4), (5).  In particular,  classes of type (4) and (5) cannot be represented
by embedded symplectic spheres for any symplectic form.

\end{theorem}

\begin{proof}  Assuming that $A$ is represented by a smooth sphere,
Lemma \ref{l:hom} gives all the classes in the list except (3).
This last class follows from the results of Section \ref{s:-4red},
and all other possibilities are excluded.

The class $H-E_1-..-E_5$ can clearly be represented by an embedded
symplectic sphere for some symplectic form $\omega$ with
$K_\omega=K_{st}$.   The class $-H+2E_1-E_2$ can be viewed as the
blow-up of a section in a Hirzebruch surface.


We will show in Section \ref{s:Con-4} that $-a(-3H+\sum_{i=1}^9
E_i)-2E_{10}$ can be represented by symplectic spheres for some
symplectic forms, hence also has smooth representatives.  This is a
slight overkill: a smooth representative of this class could be
constructed directly by the circle sum construction in \cite{LL3}.
We leave that for interested readers.

$K_{st}\cdot A=2$ is clear in (1),(2) and (3).  Choosing $0<\epsilon_i<<1$
appropriately, a symplectic form in the class $aH-\sum
\epsilon_iE_i$ has the standard canonical class and pairs positively
with $A$ in these cases.

Now we analyze the classes of type (4) and (5).  The class $2E_1$ is
smoothly representable by a sphere: Consider a smooth sphere in the
class $E_1$. A small pushoff of this sphere produces a second
exceptional sphere in the same class intersecting once, a smoothing
of this will produce a smooth sphere in the class $2E_1$.


Notice that $K_{\tau}\cdot 2E_1=2$ implies that $-E_1$ can be
represented as a $\tau$-symplectic sphere from Theorem A of
\cite{TJLL2}. But this means $[\tau]\cdot (2E_1)<0$. The argument is valid
for $2(H-E_1-E_2)$, by noticing that $H-E_1-E_2$ is also an
exceptional class.


\end{proof}

\brmk\label{r:algorithm} For completeness, we describe the
explicit algorithm producing necessary $D(M)$-equivalences for
$A=aH-\sum b_iE_i$ throughout this section, regardless of the value of
its square. With such an algorithm, one may determine in a finite number of
steps whether a given homology class is represented by a smooth or
symplectic sphere given the theorems proven here.  This procedure is
implicit in \cite{LL2} and the proof of Lemma \ref{l:hom}.

\begin{enumerate}

\item If $a<0$,  just change it to  $-a$  using   reflection along  the
$+1$ sphere $H$.

\item If $b_i < 0$, change it to $-b_i$ using reflection
along the $-1$ sphere $E_i$.

\item Arrange $b_i \geq b_{i+1}$  using reflections along $E_i -
E_{i+1}$.

\item Reflect along $H-E_1-E_2-E_3$.

\item Repeat the above process until one arrives at $k=2$, $k=3$
or $b_1^2+b_2^2+b_3^2-3\le a^2\le \frac{3}{4}(b_1^2+b_2^2+b_3^2)$.
Proceed as in Lemma \ref{l:hom}.

\end{enumerate}

\ermk

\subsection{Tilted transport: constructing symplectic (-4)-spheres}\label{s:Con-4}

\subsubsection{Reduction from Theorem \ref{stability}}
We now consider the existence part of Theorem \ref{t:spheres}. From the
assumptions we may assume $A$ has the form specified as type (1),(2)
or (3) in the list of Theorem \ref{t:-4dmequiv}.

We consider type (1).  By applying an appropriate diffeomorphism we
assume $A=-H+2E_1-E_2$.  Since $D(M)$ acts transitively on the set
of symplectic canonical classes (Lemma \ref{trans}), they are all of
the form $\pm 3H+\sum \pm E_i$.  Any canonical class with $K_\w\cdot
A=2$ must have $K_\w=-3H+E_1+E_2+\sum_{i\ge3}\pm E_i$.  Applying
trivial transforms on $E_i$ for $i\ge 3$ will not affect the pairing
$K_\w\cdot A$ or $\w\cdot A$.  Hence we may assume $K_\w=K_{st}$.
Any symplectic form with canonical class $K_{st}$ are deformation
equivalent by \cite{MP94}, hence it suffices to construct a
symplectic sphere for \textbf{some} symplectic form associated to
$K_{st}$ when $A$ is precisely the class $-H+2E_1-E_2$.

For type (2) we again assume $A=H-E_1-\dots -E_5$.  In this case
$K_\w$ may be one of the following: it either equals
$-3H+E_1+\dots+E_5+\sum_{i\ge5}\pm E_i$, or
$3H-E_1-E_2-E_3+E_4+E_5+\sum_{i\ge5}\pm E_i$ up to reordering the
first five exceptional classes.  In the latter case, we apply
further trivial transforms on $H$ and $E_{1,2,3}$ so that $A$ is
transformed into $A'=-H+E_1+E_2+E_3-E_4-E_5$. However, $A'$ is again
equivalent to $A$ by reflection along
$H-E_1-E_2-E_3$, which does not change the canonical class.  The
conclusion is $A$ can always be assumed to have the form
$H-E_1-\dots -E_5$ while the canonical class can be assumed to be $K_{st}$
simultaneously.

A similar reduction holds for type (3) and we give only a sketch: when
$A=-a(-3H+\sum_{i=1}^9 E_i)-2E_{10}$ and $K_\w\cdot A=2$, then
$K_\w=(-1)^{\delta}(-3H+\sum_{i=1}^9 E_i)+E_{10}+\sum_{i\ge 11}\pm
E_{i}$.  When $\delta=0$, the reduction works exactly as in the type (1)
case.  When $\delta=1$, again from Theorem A of \cite{TJLL2}, $-H$
and $-E_{j}$ for $j\le 9$ and $E_{10}$ are all represented by
$\w$-symplectic spheres.  This implies $\w(A)<0$ thus excluded by
our assumption.

To summarize our discussion, we have the following reduction of
Theorem \ref{t:spheres}:

\blem\label{l:FinalRed} Assume $A$ equals any one of the classes of type (1), (2) or (3) specified in Theorem \ref{t:spheres}.  If $A$ is represented by a
symplectic sphere for some symplectic form $\w$ with $K_\w=K_{st}$, then Theorem \ref{t:spheres} holds.

\elem

We would like to emphasize that $A$ is assumed to equal the classes
in Theorem \ref{t:spheres} instead of being only $D(M)$-equivalent.  Also
recall that when $K_\w=K_{st}$,  the
symplectic manifold can be assumed to be obtained by blow-ups of
symplectic $\CP^2$.

It is not difficult to verify Lemma \ref{l:FinalRed} for classes
$-H+2E_1-E_2$ and $H-E_1-..-E_5$: one may choose a symplectic form
$\w$ where $\w(E_i)$ are small enough, then the former class has
representatives as iterated blow-ups from an $H$-sphere.  For the
latter class, by a change of basis, they are the class $F-2S$ in
$S^2\times S^2\#(k-1)\overline{\CP^2}$ which clearly has a
symplectic representative (here $F$ and $S$ denotes the fiber and
base homology classes in $S^2\times S^2$).  Therefore, the following
lemma implies Theorem \ref{t:spheres}:

\blem\label{l:reduction2b} The class $-a(-3H+\sum_{i=1}^9
E_i)-2E_{10}$ has an $\w$-symplectic representative for some $\w$
with $K_\w=K_{st}$.

\elem

The proof of this lemma will occupy the rest of this section.


\subsubsection{The Tilted Transport}

We start our discussion in a more general context.  Let
$\pi:(E^{2n},\w_E)\rightarrow D^2$ be a symplectic Lefschetz
fibration.  This means:

\begin{itemize}
  \item $(E,\w_E)$ is a symplectic manifold with boundary $\pi^{-1}(\partial
  D^2)$;
  \item $\pi$ has finitely many critical points $p_0,\dots,p_n$ away
  from $\partial D^2$, while $\pi^{-1}(b)$ is a closed symplectic
  manifold symplectomorphic to $(X,\w)$ when $b\neq \pi(p_i)$ for any
  $i$.
  \item Fix a complex structure $j$ on $D^2$.  There is
  another complex structure $J_i$ defined near $p_i$, so that $\pi$
  is $(J,j)$-holomorphic in a holomorphic chart $(z_1,\dots,z_n)$ near $p_i$, and under this chart,
  $\pi$ has a local expression $(z_1,\dots, z_n)\mapsto
  z_1^2+\dots+z_n^2$.

\end{itemize}

Take a regular value of $\pi$, $b_0\in D^2$ as the base point.
Suppose one has a submanifold $Z^{2r-1}\subset \pi^{-1}(b_0)$. We
say $Z$ has \textit{isotropic dimension 1} if at each $x\in Z$,
$(T_xZ)^{\perp\w}\cap T_xZ=\R\langle v_x\rangle$.  We call $v_x$ an
\textit{isotropic vector} at $x$.  A simple example of a submanifold
of isotropic dimension 1 is a contact type hypersurface.  A special
case more relevant to us is a closed curve on a surface.

Suppose we have a (based) Lefschetz fibration $(E,\pi, b_0)$ with a
submanifold $Z\subset \pi^{-1}(b_0)$ of isotropic dimension 1.  Let
$\gamma(t)\subset D^2$ be a path with $\gamma(0)=b_0$.  Assume
$\gamma(t)\neq \pi(p_i)$ for all $t$ and $i$.  Notice there is a natural
symplectic connection on $E$ in the complement of singular points as
a distribution: for $x\in E\backslash \coprod_{i=1}^n \{p_i\}$, the
connection at $x$ is defined by $(T^vE)_x^{\perp\w}$ .  Here $T^vE$
is the subbundle of $TE$ defined by vertical tangent spaces
$T(\pi^{-1}(\pi(x)))$ at point $x$.  $E|_\gamma=\pi^{-1}(\gamma)$
  thus inherits this
connection  and thus a trivialization by parallel transports.  The
symplectic connection also defines a unique lift of $\gamma'$ to a
vector field of $E|_\gamma$.  We will use $\pi^{-1}(\gamma')$ to
represent this lift.

Now choose a vector field $V$ on $E|_\gamma$ tangent to the fibers; one
obtains a flow defined by $V+\pi^{-1}(\gamma')$.  Suppose the
following holds:

\begin{condition}\label{cond:c}\hfill
\begin{itemize}
  \item  $Z_t\subset E_{\gamma(t)}$ is the time $t$-flow of $Z_0=Z$,
  and each $Z_t$ is of isotropic dimension 1.
  \item For any $x_t\in Z_t$, let $v_{x_t}$ be the isotropic vector.  Then $\w(V,v_{x_t})\neq0$.
\end{itemize}

\end{condition}

Let $\widehat Z=\coprod Z_t$.  It is then easy to see that $\wh Z$
is a symplectic submanifold of $E$ with boundary on $E_{\gamma(0)}$
and $E_{\gamma(1)}$.  We call $\wh Z$ a \textit{tilted transport} of
$Z$.  A special case is when $\gamma(0)=\gamma(1)$ and $Z_0=Z_1$. In
such cases, one could be able to adjust $V$ appropriately so that
$\wh Z$ is a smooth closed symplectic submanifold, which we will
call a \textit{tilted matching cycle}. In the following we sometimes write
$\wh Z^\gamma$ and $Z^\gamma_t$ to emphasize the dependence on the
path $\gamma$.

\brmk
One notices that the symplectic isotopy class of the
tilted transport is independent of specific choices inside an isotopy classes of the
auxiliary data ($\gamma(t), V$ etc.).  However, it could happen that
the auxiliary data form a space with more than one component, e.g.
when the fibers are of dimension 2, then the choice of $V$ has at
least 2 connected components.   Moreover, in general there is no guarantee for a
Hamiltonian isotopy instead of  a symplectic isotopy.

\ermk

\brmk The tilted transport construction as described here can be
easily generalized in many ways.  A most interesting generalization
is that one could admit $V$ with singularities thus change the
topology of $Z_t$ when $t$ evolves.  We will explore further
applications of such constructions in upcoming work.\ermk

\subsubsection{Construction of (-4)-spheres}

Let us now specialize the tilted transport construction to the case of $-4$-spheres in the class
$-a(-3H+\sum_{i=1}^9 E_i)-2E_{10}$.  We will continue to use the
notation in the previous section.

Take the usual Lefschetz fibration by elliptic curves on
$E(1)=\CP^2\#9\overline\CP^2$, which can be endowed with a K\"ahler
form $\w$ compatible with the fibration structure. It suffices to
restrict the fibration to a neighborhood of a singular fiber,
yielding a fibration over $0\in D=D^2(2)\subset \C$, $1$ being the
unique critical value and the generic fiber an elliptic curve.
Denote by $p_0$ the unique critical point of this fibration.

Take $0\in D$ as the base point $b_0$.  From the usual construction
of vanishing cycles, there is a circle $C\subset \pi^{-1}(0)$ which
has the following property.  Consider $\gamma:[0,1]\rightarrow D$,
$\gamma(t)=t$, then:

$$\lim_{t\rightarrow 1}\phi_t(y)=p_0\Longleftrightarrow y\in C.$$

Here $\phi_t$ is the parallel transport using the induced symplectic
connection along $\gamma(t)$.  Let $\pi^{-1}(0)$ be identified with
a symplectic $\T^2=S^1\times S^1$, where the two $S^1=\R/\Z$ factors
are parametrized by $s,r\in [0,1]$, and $C$ is identified with
$\{r=0\}$. Take a neighborhood of $C$ as $S^1\times
[-\delta,\delta]\subset \T^2=\pi^{-1}(0),\delta\ll 1$. Assume
without loss of generality also that the symplectic orientation is
given by $\partial_s\wedge\partial_r$, i.e.
$\w(\partial_s,\partial_r)>0$. We propagate this coordinate to
$E_\gamma\backslash p_0$ by the symplectic connection.  Let
$V=\frac{\delta}{2}\cdot\partial_r$, we define a tilted transport of
$C_0=S^1\times\{-\delta\}$ from $E_0$ to $E_1$, denoted as
$\Sigma_0'$.  Now $C_0^\gamma=\partial(\Sigma_0')\cap E_1$ bounds
two symplectic disks on $E_1$ (which is a fish-tail), but only one
of them concatenates by the correct orientation with $\Sigma_0'$.
Concretely, this disk is precisely the image of the usual parallel
transport of $S^1\times [0,\frac{\delta}{2}]$ to $E_1$.  We denote
this symplectic disk on $E_1$ as $\Sigma_0''$.  A suitable smoothing
of $\Sigma_0'\cup\Sigma_0''$ yields a symplectic disk $\Sigma_0$
with boundary $C_0$ on $E_0$.  This is a symplectic variant of the usual
vanishing thimble construction.

Now choose another embedded curve $\bar\gamma(t)\subset D^2$ so that

\[\left\{ \begin{aligned}
         \bar\gamma(0)=0,\bar\gamma(1)=1; \\
          \bar\gamma'(0)=-\gamma'(0);\\
          \gamma\cap\bar\gamma=\{0,1\}.
                          \end{aligned} \right.
                          \]

Again $E_{\bar\gamma}$ inherits a symplectic connection thus a
trivialization in the complement of the singular fiber and we can
trivialize this part of the pull-back fiber bundle using the
symplectic connection and parametrize it by coordinate
$(s,r,t)\subset S^1\times S^1\times [0,1)$ as before.  We may adjust
the fibration appropriately over $\bar\gamma$ so that the vanishing cycle along
$\bar\gamma$ is again $C=S^1\times \{0\}\times \{0\}$.  Let
$W_a=-(a-2\delta)\partial_r$.  The tilted transport associated to
$\bar\gamma$ and $W_a$ starting from $C_0$ gives a symplectic annuli $\Sigma_1'$ with
boundary $C_0$ and $C^{\bar\gamma}_1\subset E_1$.  $C^{\bar\gamma}_1$
again bounds two symplectic disks on the singular fiber, and we can
take the image of $S^1\times [0,\delta]$ under the (usual) parallel
transport $\Sigma_1''$.  The union $\Sigma_1'\cup\Sigma''_1$ again
forms a symplectic disk with boundary $C_0$ after smoothing.

Now the union $\Sigma_0\cup\Sigma_1$ then matches to form a
smoothly immersed symplectic $S^2$, with adjustments on $V$ and
$W_a$ near $C$ if necessary.  This symplectic $S^2$ is denoted as
$\Sigma$, and it has a unique double point at $p_0$ and is embedded
otherwise. It is not hard to see from our construction of $W_a$ that
$[\Sigma]=-aK$ for $K$ being the Poincare dual of the canonical class of
$(E(1),\omega)$, which is homologous to a fiber class: by resolving the
self-intersection at $p_0$, one has an embedded surface in $E(1)$
which is smoothly isotopic to a multiple cover of a generic fiber of
the fibration.  One may then perform a small symplectic blow-up at
$p_0$ which resolves the self-intersection and  which yields an embedded
symplectic sphere with class $-aK-2E_{10}$ for any $a\ge1$ in $\CP^2\#10\overline{\CP^2}$ for an appropriate
symplectic form .

This proves Lemma \ref{l:reduction2b} and hence the proof of Theorem \ref{t:spheres} is complete.

\subsection{Spheres with self-intersection $-1$,$-2$ and $-3$\label{s:-1-2-3}}

To begin, we note that, when $b^-(M)=0$, there are no spheres of
negative intersection. When $b^-(M)=1$,  the only rational manifolds
are  $S^2\times S^2$ and $\mathbb CP^2\#\overline{\mathbb CP^2}$.
Due to the existence of an orientation reversing diffeomorphism, the
negative square case can be reduced to the positive square case. The
minimal genus of $A\in H_2(M,\mathbb Z)$ in these cases has been
determined in \cite{R} and from this all symplectic spheres can be
determined.

So we generally assume in the following that $b^-(M)\geq 2$.

\subsubsection{Spheres with Self-Intersection $-1$,$-2$}

Spheres with square $-1$ are exceptional spheres. They are all
$D(M)$-equivalent to either $E_1$ or $H-E_1-E_2$ from \cite{LL2}.

We now consider spheres with self-intersection $-2$.  A
classification of smooth $-2$-spheres can be found in \cite{LL2}.  For rational manifolds $M$,  Lagrangian $-2$-spheres have only recently been
classified in \cite{LW}.

\begin{prop}[Lemma 3.4, Lemma 3.6, \cite{LL2}]\label{2-standard}Let $M $ be a rational manifold.  Assume that  $b^-(M)\geq 2$.  Let $A\in H_2(M,\mathbb Z)$ with $A^2=-2$.  Assume that $A$ is represented by a smoothly embedded sphere.  Then up to
the action of $D(M)$, $A$ is one of the following:

\begin{enumerate}
\item If $A$ is characteristic, then $b^-(M)=3$ and $A=H-E_1-E_2-E_3$.
\item If $A$ is not characteristic, then $A= E_1-E_2$.
\end{enumerate}

\end{prop}

This proves one aspect of Speculation \ref{spec:isphere} for spheres with $A\cdot A=-2$.  The following completes Speculation \ref{spec:isphere} in the $-2$ case for rational manifolds.

\begin{prop}
Let $(M,\omega)$ be a symplectic rational manifold  and $A\in H_2(M,\mathbb Z)$ such that $A\cdot
A=-2$.  Then $A$  is represented by a $\omega$-symplectic sphere for some
symplectic form if and only if\begin{enumerate}
\item $g_\omega(A)=0$,
\item $[\omega]\cdot A>0$ and
\item  $A$  is represented by a smooth sphere
\end{enumerate}

Moreover,   when $b^-(M)\ne 3$, $V$ can be chosen to be the blow-up of an exceptional sphere.
\end{prop}

\begin{proof}  Assume that $A$ is represented by a smooth sphere.  Then $A$ is $D(M)$ - equivalent to one of the classes in
Proposition \ref{2-standard}.

There exists a symplectic form $\tau$ with $K_\tau=K_{st}$ such that the classes from Prop. \ref{2-standard} are represented by a $\tau$-symplectic  surface.  This can be obtained through an
appropriate blow-up from $\mathbb CP^2$.

The result now follows from Lemma \ref{l:3.11}.

\end{proof}

\subsubsection{Spheres with Self-Intersection $-3$}

We proceed as in the $-4$-case.

\begin{lemma}
Let $M=\mathbb CP^2\#k\overline{\mathbb CP^2}$ with $k\ge 1$ and $A\in H_2(M,\mathbb Z)$.  Assume that $A\cdot A=-3$.  Then, up to $D(M)$-equivalence, $A$ is a reduced class, $-H+2E_1$  or $H-E_1-..-E_4$.
\end{lemma}

\begin{proof}For $k=1$, $a^2-b_1^2=-3$ allows only for $\pm H\pm 2E_1$.

For $k=2$, again Lemma 1, \cite{LL4}, reduces the problem to classes with $2a\le b_1+b_2$.  This produces no further classes beyond the one above.

For $k\ge 3$, as in Lemma 3.4, \cite{LL2}, it can be shown using reflections along $-2$-spheres $H-E_i-E_j-E_k$ that either $A$ is reduced or
\[
b_1^2+b_2^2+b_3^2-3\le a^2\le \frac{3}{4}(b_1^2+b_2^2+b_3^2).
\]
In addition to this inequality, $a^2-\sum b_i^2=-3$ and $b_1\ge b_2\ge ...\ge b_k\ge 0$.  The solutions to this system, written in short as $(a,b_1,b_2,..,)$, are:
\[
(0,1,1,1,0..), (1,1,1,1,1,0,..),  (1,2,0,..),
\]
\[
(2,2,1,1,1,0,..), (3,2,2,2,0,..).
\]
Under the $D(M)$-action,  $(1,1,1,1,1,0,..)$ is in a class of its own when $k=4$. The other classes are all equivalent, and when $k\ge 5$, $(1,1,1,1,1,0,..)$ is included as well.

\end{proof}

\begin{theorem}\label{t:-3smooth}Let $M=\mathbb CP^2\#k\overline{\mathbb CP^2}$ and $k\ge 10$.  Then there exists no reduced class $A\in H_2(M,\mathbb Z)$ with $A\cdot A=-3$  and $K_{st}\cdot A\le 1$.
\end{theorem}

\begin{proof}
This result follows by repeating the proof of Theorem \ref{t:-4smooth} in the $\tau=1$ case.  This involves no new methods.
\end{proof}

For $k\ge 3$, $-H+2E_1$ is equivalent to $E_1-E_2-E_3$.  This class
is represented by a symplectic sphere.

Consider a symplectic representative $V$ of  $-H+2E_1$ in $\mathbb
CP^2\#3\overline{\mathbb CP^2}$.  Blow up one point of this
representative to obtain a symplectic $-4$-sphere $Z$ in $\mathbb
CP^2\#4\overline{\mathbb CP^2}$ in the class $-H+2E_1-E_4$.  By Cor
3.3 in \cite{BLW} there exist 3 disjoint exceptional spheres from
$Z$.  In particular, two of these exceptional spheres are just $E_2$
and $E_3$, both of which are not affected by blowing down $E_4$.
Thus after blowing down $E_4$, we re-obtain $V$ as a symplectic
manifold, but also obtain two pairwise disjoint exceptional spheres
which are furthermore disjoint from $V$.  Hence blowing down $E_2$
and $E_3$ leaves $V$ unchanged, thus providing for a symplectic
$-3$-sphere in $\mathbb CP^2\#k\overline{\mathbb CP^2}$ for $k=1,2$.
Alternatively, the  class  $-H+2E_1$ can be related to a Hirzebruch
surface; it can be viewed as a section in the non-trivial $S^2$
bundle over $S^2$.  This is symplectic for an appropriate choice of
symplectic form.



We have the analogues of the results in the $-2$ and $-4$ case.

%
%
%

\begin{theorem}\label{3-standard}
Let $(M,\omega)$ be a rational symplectic manifold and  $A$  a homology
class with $A^2=-3$.

  Then $A$  is represented by a $\omega$-symplectic sphere if and only if\begin{enumerate}
\item $g_\omega(A)=0$,
\item $[\omega]\cdot A>0$ and
\item  $A$  is represented by a smooth sphere
\end{enumerate}

Moreover,   when $b^-(M)\ne 4$, $V$ can be chosen to be the blow-up of an exceptional sphere.
\end{theorem}

\subsection{Discussions}

We conclude this section on spheres in rational surfaces by indicating some possible
directions extending further our results.

\subsubsection{Spheres with large negative square}

The tilted transport is used to prove that the classes of type (3)
in Theorem \ref{t:-4dmequiv} are representable by symplectic
spheres.  While the focus of the previous section was to construct
$(-4)$-spheres, the results also prove the existence of highly
singular curves in $\mathbb CP^2$, as we will explain.

Using Lemma \ref{rel exceptional}, there exist exceptional spheres
in the classes $E_i$ which intersect the $(-4)$-sphere $V_a$ of type
(3) locally positively and transversally and which can be blown down
to produce a point of $a$-fold intersection.  Doing this for all ten
exceptional spheres $E_1,..,E_{10}$ produces a curve $C\subset
\mathbb CP^2$ in the class $3aH$ with one nodal point and 9 points
of a-fold self-intersection.

Applying Prop. 3.3, \cite{LU} to one of the a-fold intersections, one
can successively perturb away intersection components to make an
a-fold intersection into a ($a-1$)-fold intersection and ($a-1$) double
points.  Repeating this will produce singular curves with differing
combinations of intersections.  Blowing up at the intersection points will produce
an embedded sphere in some rational manifold.  This result is
summarized in the following lemma.

To each a-fold singular point, associate the value $k_i$ ($1\le i\le 9$), $0\le k_i\le a-2$ and $a>2$, describing the number of curves which have been perturbed out of the singular point.  Let
\[
N_i=k_i(a-\frac{k_i}{2}-\frac{1}{2})
\]
and $m=10+\sum_{i=1}^9N_i$.

\begin{lemma}
In $\mathbb CP^2\# m\overline{\mathbb CP^2}$ the class \[A=3aH-\sum_{i=1}^9(a-k_i)E_i-2E_{10}-2\sum_{i=1}^{\sum N_i}E_i\] is represented by an embedded connected symplectic sphere.
\end{lemma}

Then
\[
A\cdot A=\sum_{i=1}^9[k_i^2-2k_i(a-1)]-4.
\]
This allows us to note the following interesting examples.
\begin{enumerate}
\item Consider the two classes
\[
A_1=12H-4\sum_{i=1}^6E_i-3(E_7+E_8+E_9)-2E_{10}-2\sum_{i=11}^{19}E_i-E_{20}
\]
corresponding to $a=4$, $k_1=k_2=k_3=1$, $k_i=0$ otherwise and then one point blown up and
\[
A_2=12H-4\sum_{i=1}^7-2(E_8+E_9)-2E_{10}-2\sum_{i=11}^{20}E_i
\]
corresponding to $a=4$, $k_1=k_2=2$, and $k_i=0$ otherwise.  Both
classes can be represented by embedded symplectic spheres of
self-intersection $-20$ in $\mathbb CP^2\#20\overline{\mathbb
CP^2}$.  Note that both classes are reduced, by the uniqueness of
reduced form \cite{LL2}, they are not $D(M)$-equivalent.

\item It is very simple to construct spheres with
self-intersection $-l$ in $\mathbb CP^2\# m\overline{\mathbb CP^2}$
for some $l>m$.  For example, the curve
\[
A_3=9H-2\sum_{i=1}^{28}E_i
\]
corresponding to $a=3$ and $k_i=1$ has $A_3\cdot A_3=-31$ and lies in $\mathbb CP^2\# 28\overline{\mathbb CP^2}$  or
\[
A_4=12H-4\sum_{i=1}^5E_i-3(E_6+E_7+E_8+E_9)-2E_{10}-2\sum_{i=11}^{22}E_i
\]
corresponding to $a=4$, $k_1=..=k_4=1$ and $k_i=0$ otherwise has $A_4\cdot A_4=-24$ and lies in $\mathbb CP^2\# 22\overline{\mathbb CP^2}$.  Compare this with the lower bound for spheres in irrational ruled manifolds obtained in Lemma \ref{irr}.

\end{enumerate}


\subsubsection{A local variant of tilted transport and
symplectic circle sum}

We explain next how to use a rather simple case of tilted transport
to partly recover the circle sum construction in symplectic
geometry. Note the corresponding counterpart is well-known in the
smooth category.

The setting under consideration is a pair of disjoint symplectic
surfaces $S_0, S_1\subset (M^4,\w)$.  Suppose one has an open set
$U\subset M$ so that $U\cong S^1\times [-1,1]\times D^2(2)$, a
trivial bundle over $D^2$ with annulus fibers, while $S_i\cap
U=S^1\times [-1,1]\times \{i\}$.  We claim that there is an embedded
symplectic surface $S$ which is the circle sum of $S_0$ and $S_1$.

The question is local so we concentrate on the trivial bundle $U$.
Remove the part $S^1\times [-\frac{1}{2},\frac{1}{2}]$ from the
fibers $U_0$ and $U_1$.  Consider two embedded arcs $\gamma(t)$ and
$\bar\gamma(t)$, which only intersect at
$\gamma(0)=\bar\gamma(0)=0$ and $\gamma(1)=\bar\gamma(1)=1$.  By
choosing $W$ appropriately on $U_\gamma$, one easily constructs a
tilted transport which concatenates $S^1\times
[-1,-\frac{1}{2}]\times \{0\}$ with $S^1\times [\frac{1}{2},1]\times
\{1\}$.  Similarly one concatenates $S^1\times
[-1,-\frac{1}{2}]\times \{1\}$ with $S^1\times [\frac{1}{2},1]\times
\{0\}$ by choosing another tilted transport on $\bar\gamma$.  This
realizes the circle sum as claimed.

As immediate consequence of the construction, by taking a finite
number of nearby copies of generic fibers in an arbitrary Lefschetz
fibration of dimension $4$, one realizes $n[F]$ as an embedded
symplectic surface by performing symplectic circle sums on two
consecutive copies.  This construction clearly generalizes to
higher dimensions in appropriately formulated cases, which is left
to interested readers.  This particular case applied to situations
in Section \ref{s:Con-4} yields an alternative proof for Lemma
\ref{l:reduction2b}.

\section{Spheres in irrational ruled manifolds}\label{class}



%

Let $M$ be an irrational ruled symplectic manifold.  Then the
minimal model of $M$ is an $S^2$-bundle over a surface $\Sigma_h$
with $h\ge 1$.  Recall that, when $M$ is minimal, there can be no
negative symplectic spheres from adjunction (see also the proof of
Lemma \ref{irr}), so there is nothing to investigate.  In the
non-minimal case, the blow-up of the trivial bundle and the blow-up
of the non-trivial bundle are diffeomorphic, we fix a standard
representation: Let $M=(\Sigma_h\times S^2)\#k\overline{\mathbb
CP^2}$.  Denote by $\{S,F,E_1,..,E_k\}$ the standard basis of $M$,
where $S$ denotes the class of a surface of genus $h$ and $F$ is the
fiber class. Denote the standard canonical class
$K_{st}=-2S+(2h-2)F+\sum E_i$.

\subsection{\label{irrat}Smooth spheres}

\begin{lem} \label{irr} Let $M$ be an irrational ruled manifold with $b^-(M)=k$, and   $A\in H_2(M,\mathbb Z)$ is a class with  $A^2=-l$ with $l\geq 1$.  Assume that $A$ satisfies the following:
\begin{enumerate}
\item $A$ is represented by a smoothly embedded sphere;

\item $A\cdot A =-l$;
\item  $A\cdot [\omega]>0$ for some symplectic form  $\omega$ with $K_{\omega}=K_{st}$ and

\item  $g_{K_{st}}(A)\ge 0$.

\end{enumerate}
Then   $$l\leq k.$$  Moreover, up to permutations of $E_i$,
\begin{equation} \label{A} A=bF+\sum_{i=1}^{1-b} E_i -\sum_{j=2-b}^l E_j,
\end{equation}
for some $b$ with $b\leq 1$.

\end{lem}

\begin{proof} Write the  class $A=aS+bF+\sum c_iE_i$.  Since the projection of any
smooth sphere  representing $A$ to the base in an irrational ruled
manifold is null homotopic, we must have $a=0$ from condition (1) of
Lemma \ref{irr}. This in particular shows that $S^2\times \Sigma_h$
and $S^2\tilde\times \Sigma_h$ admit no spheres of negative
self-intersection for $h\ge 1$.

Conditions (2) and (4) of Lemma \ref{irr} imply that
\[
\sum c_i^2=l\mbox{  and  }2b+\sum c_i\le 2-l
\]
which can be combined to give
\begin{equation}\label{2}
2b+\sum c_i(c_i+1)\le 2.
\end{equation}

\begin{lem}   \label{ci}Let $A=bF+\sum c_iE_i$ and  assume that $A$ satisfies the last 3 conditions of Lemma \ref{irr}.   Then $|c_i|\le 1$ for all $i$ .

\end{lem}

\begin{proof}
For $n=1$ it  is well known that $A= E_i$ or $F-E_i$ (see e.g
\cite{L1}, \cite{TJLL}). In fact, this classification provides
constraints on  symplectic forms with canonical class $K_{st}$.
  Suppose  $\omega$ is  a symplectic form with $K_{\omega}=K_{st}$ and
  $[\omega]=cS+dF+\sum e_iE_i$. Then since $E_i$ and $F-E_i$ are $\omega$-exceptional
  classes, we have

\begin{equation}\label  {symplectic form coeff}  c>|e_i|, \quad e_i<0.\end{equation}

{\bf Case 1:}  Assume that $b\ge 0$ and $|c_i|\ge 2$ for some $i$.
Each term on the left of \ref{2} is non-negative by assumption.  In
particular, $c_i(c_i+1)>2$ unless $c_1=-2$, $c_j=0$ for $j\ge 2$ and
$b=0$.  However, for symplectic forms $\omega$ satisfying
\eqref{symplectic form coeff}, the class $-2E_1$ is not
symplectically representable, i.e. $[\omega]\cdot(-2E_1)<0$.

{\bf Case 2:}  Assume that $b< 0$ and $|c_i|\ge 2$ for some $i$.
Without loss of generality let $c=1$.  As $[\omega]\cdot A>0$ we
have
\[
b-\sum e_ic_i>0\;\;\Rightarrow\;\;0>b>\sum e_ic_i.
\]
for some $c, e_i$ satisfying \eqref{symplectic form coeff}. In
particular, for some $i$ it must hold that $c_i>0$.  Then from
$|e_i|<1$,
\[
0<b-\sum e_ic_i<b+\sum_{c_i>0}c_i+\sum _{c_i<0}|c_i|e_i
\]
and thus in particular
\[
0<b+\sum_{c_i>0}c_i\mbox{  and  }0<b+\sum_{c_i>0}c_i^2.
\]
This can be used to rewrite \eqref{2} as

\begin{equation}\label{longesti}
2\ge \underbrace{b+\sum_{c_i>0}c_i}_{>0}+\underbrace{b+\sum_{c_i>0}c_i^2}_{>0}+
\underbrace{\sum_{c_i<0}(c_i^2+c_i)}_{\ge0}.
\end{equation}

If there exists an $i$ such that $c_i<-2$, then $c_i^2+c_i>2$.  For
$c_i=-2$ we must have $b=0$, this case has been considered
previously.

Thus we may assume that $c_i\ge -1$.  In this case, the last term in
\eqref{longesti} vanishes.  As the first two terms must be positive,
they must both be equal to 1.  Thus
\begin{equation} \label{b constraint}
b=1-\sum_{c_i>0}c_i=1-\sum_{c_i>0}c_i^2
\end{equation}
which can be rewritten as
\[
\sum_{c_i>0}(c_i^2-c_i)=0.
\]
This implies $c_i=1$ for those $i$ such that $c_i>0$.

We have thus shown that $c_i\in\{-1,0,1\}$ for all $i$.

\end{proof}

We now complete the proof of Lemma \ref{irr}.  The claim follows by
reindexing $\{E_i\}$ so that
$$(c_i)=(1,\cdots, 1, -1\cdots, -1, 0\cdots)$$
and applying \eqref{b constraint} to calculate $b$.
\end{proof}

The relation between $b$ and $c_i$ given by \ref{b constraint} leads
to the following three possibilities:
\begin{enumerate}
\item $b>0$: Then \ref{b constraint} implies that $b=1$ and
\[
A= (F-E_1)- \sum_{j=2}^l E_j.
\]
\item $b=0$: Then there is a unique index with $c_i=1$.  Thus
\[
A=E_1-\sum_{j=2}^lE_j.
\]
\item $b<0$: Rewrite $A$ as follows:
\[
A=-|b|F+\sum_{i=1}^{1-b} E_i -\sum_{j=2-b}^l
E_j=E_1-\sum_{i=2}^{1-b}(F-E_i)-\sum E_j.
\]
\end{enumerate}
Notice that in all cases, we can view $A$ as having the class of the
blow-up of an exceptional sphere.

\subsection{Symplectic spheres}

The stability results of Section \ref{nodal} now allow us to confirm
the existence of a symplectic curve for a choice of symplectic form.
In particular, we may fix a choice of symplectic canonical class.
This leads to:

\begin{prop}\label{irr-blowdown}
Given any symplectic form $\omega$ with $K_{\omega}=K_{st}$,  and a
class $A\in H_2(M,\mathbb Z)$ of the form \eqref{A} up to
permutations of $E_i$, $A$ is represented by $\omega-$symplectic
sphere $V$ if and only if $A\cdot [\omega]>0$.

\end{prop}

\begin{proof}
To begin, we show that $A$ as in  \eqref{A} is $D(M)$-equivalent to
either $E_1-\cdots -E_i$ or $F-E_1-\cdots -E_j$.  To see this, we
consider different values of $b$:

When $b=1$, $A$ is already in the desired form.

If $b=0$, then  $A=E_1-   E_2  - \cdots -E_l $.  If $A$ is characteristic then we are done.  Otherwise,  use  the reflection along $F-E_1- E_p$ to get to
$F-E_p -E_2 -\cdots -E_l$.
Here $p=b^-(M)-1$. Notice that $p > l$ in this case.

If $b= - 1$, then  $A= - F+E_1+E_2 - E_3 \cdots$.
Now reflect along $F-E_1-E_2$ to transform it to $F-E_1-E_2 -E_3 \cdots$.

When $b = -2$,  $A= -2F +E_1+E_2+ E_3 -E_4 \cdots$. Then reflection along $F-E_1-E_2$
transforms it into  $E_3 - E_1 -E_2 -  E_4 \cdots$.

For $b<-2$ use reflection along $F-E_1-E_2$ and induction to achieve this transformation.

Now assume that $A\cdot [\omega]>0$.  Transform the class $A$  via a diffeomorphism  $\phi$ to one of the forms above. Notice that all of the transformations above preserve the canonical class.  Thus the new class $\phi_*A$ is associated to some symplectic form with standard canonical class.  Note that both $E_1-\cdots -E_i$ and $F-E_1-\cdots -E_j$ can be represented by symplectic spheres obtained by an appropriate sequence of blow-ups for such a symplectic form.  Now apply $\phi^{-1}$ to get the desired result.

\end{proof}

\begin{cor}\label{irr-D}
Let $M$ be an irrational ruled symplectic manifold and  $A\in
H_2(M,\mathbb Z)$  a homology class with $A\cdot A=-l$. Then $A$  is
represented by a connected embedded symplectic sphere for some
symplectic form if and only if
 $A$  is $D(M)-$equivalent to one of the classes in Proposition \ref{irr}.

\end{cor}

\begin{proof}
Assume that $A$  is represented by a connected embedded
$\omega$-symplectic sphere for some symplectic form $\omega$.  The
transitive action of $D(M)$ provides a connected embedded
$\omega'$-symplectic sphere for some symplectic form $\omega'$ with
$K_{\omega'}=K_{st}$ in some class $A'$ $D(M)$-equivalent to $A$.

Notice that $g_\omega'(A')=0$ and thus the result follows from Prop.
\ref{irr}.

The converse follows from Prop. \ref{irr-blowdown} and an
appropriate choice of symplectic form.

\end{proof}

\begin{theorem}
Let $(M, \omega)$ be an irrational ruled symplectic manifold and $A\in H_2(M,\mathbb Z)$.

Then $A$  is represented by an
$\omega-$symplectic sphere if and only if
\begin{enumerate}
\item $A$ is represented by a smooth sphere,
\item $g_\omega(A)=0$ and
\item $A$ pairs positively with $\omega$.
\end{enumerate}

\end{theorem}

\begin{proof}
Clearly if $A$  is represented by an
$\omega-$symplectic sphere the result follows.

Now consider the other direction.  $A=bF+\sum c_iE_i$ as in the
proof of Lemma \ref{irr}, by considering projection to the base,
thus we may assume that  $A\cdot A<0$.    Our assumptions imply that
there exists a diffeomorphism taking $A$ and $\omega$ to a class
$A'$ and symplectic form $\omega'$ with $K_{\omega'}=K_{st}$ such
that the above conditions continue to hold.  Lemma \ref{irr} and
Proposition \ref{irr-blowdown} then imply that $A'$ is represented
by a $\omega'$-symplectic sphere.  Now use the diffeomorphism to get
a $\omega$-symplectic sphere in the class $A$ (see Cor. \ref{ex}).

%


\end{proof}

This completes the proof of Theorem \ref{irratspec}.

%
%


\end{document}